\documentclass[reqno,11pt]{amsart}

\usepackage[all]{xy}
\xyoption{poly}
\usepackage{amsthm, amssymb, amsmath, amscd}
\usepackage{braket}
\usepackage{geometry}
\usepackage[dvips]{graphicx, color}

\def\a{\alpha}
\def\b{\beta}
\def\c{\gamma}
\def\d{\delta}
\def\e{\varepsilon}

\def\s{\sigma}
\def\t{\tau}

\def\NN{{\mathbb N}}
\def\PP{{\mathbb P}}

\def\ZZ{{\mathbb Z}}

\def\cal{\mathcal}

\def\cC{{\cal C}}
\def\cD{{\cal D}}

\def\cF{{\cal F}}
\def\cG{{\cal G}}
\def\cH{{\cal  H}}

\def\cP{{\cal P}}

\def\cT{{\cal T}}

\def\cV{{\cal V}}

\def\fm{{\mathfrak m}}

\def\Aut{\operatorname{Aut}}

\def\add{\operatorname{add}}

\def\Coker{\operatorname{Coker}}

\def\CM{\operatorname {CM}}

\def\dim{\operatorname{dim}}

\def\Ext{\operatorname {Ext}}

\def\gldim{\operatorname{gldim}}

\def\grmod{\operatorname{grmod}}
\def\GrMod{\operatorname{GrMod}}

\def\GrAut{\operatorname{GrAut}}

\def\H{\operatorname{H}}
\def\Hom{\operatorname {Hom}}

\def\id{\operatorname {id}}
\def\Id{\operatorname{Id}}

\def\ind{\operatorname {ind}}

\def\injdim{\operatorname{injdim}}

\def\Ker{\operatorname {ker}}

\def\lin{\operatorname{lin}}

\def\min{\operatorname{min}}
\def\mod{\operatorname{mod}}
\def\Mod{\operatorname{Mod}}

\def\Obj{\operatorname{Obj}}

\def\Proj{\operatorname{Proj}}

\def\QGr{\operatorname{QGr}}
\def\qgr{\operatorname{qgr}}

\def\rank{\operatorname{rank}}

\def\tors{\operatorname{tors}}
\def\Tors{\operatorname{Tors}}
\def\tails{\operatorname{tails}}
\def\Tails{\operatorname{Tails}}

\def\uHom{\operatorname{\underline{Hom}}}

\def\ugrmod{\underline{\grmod}}

\def\uCM{\underline {\operatorname{CM}}}

\def\<{\langle}
\def\>{\rangle}

\def\NMF{\operatorname{NMF}}
\def\uNMF{\underline{\operatorname{NMF}}}

\def\MF{\operatorname{MF}}
\def\uMF{\underline{\operatorname{MF}}}

\def\TR{\operatorname{TR}}

\def\uTR{\underline{\operatorname{TR}}}
\def\umod{\underline{\operatorname{\mod}}\,}
\def\ugrmod{\underline{\operatorname{\grmod}}\,}

\def\ind{\operatorname{Ind}}
\def\Ker{\operatorname{Ker}}

\def\rnum#1{\expandafter{\romannumeral #1}}
\def\Rnum#1{\uppercase\expandafter{\romannumeral #1}}

\theoremstyle{plain} 
\newtheorem{theorem}{Theorem}[section]

\newtheorem{lemma}[theorem]{Lemma}
\newtheorem{proposition}[theorem]{Proposition}

\theoremstyle{definition}
\newtheorem{definition}[theorem]{Definition}
\newtheorem{example}[theorem]{Example}

\theoremstyle{remark}

\newtheorem{remark}[theorem]{Remark}

\numberwithin{equation}{section}

\begin{document}

\pagenumbering{arabic}

\title[Noncommutative Kn\"orrer's Periodicity Theorem]
{Noncommutative Kn\"orrer's Periodicity Theorem and Noncommutative Quadric Hypersurfaces}

\author{Izuru Mori}

\address{Department of Mathematics,
Faculty of Science,
Shizuoka University,
836 Ohya, Suruga-ku, Shizuoka 422-8529, Japan}
\email{mori.izuru@shizuoka.ac.jp}

\author{Kenta Ueyama}

\address{Department of Mathematics,
Faculty of Education,
Hirosaki University,
1 Bunkyocho, Hirosaki, Aomori 036-8560, Japan}
\email{k-ueyama@hirosaki-u.ac.jp}



\thanks{The first author was supported by JSPS Grant-in-Aid for Scientific Research (C) 16K05097 and JSPS Grant-in-Aid for Scientific Research (B) 16H03923.
The second author was supported by JSPS Grant-in-Aid for Early-Career Scientists 18K13381.}

\begin{abstract}
Noncommutative hypersurfaces, in particular, noncommutative quadric hypersurfaces are major objects of study in noncommutative algebraic geometry. 
In the commutative case, Kn\"orrer's periodicity theorem is a powerful tool to study Cohen-Macaulay representation theory since it reduces the number of variables in computing the stable category $\uCM(A)$ of maximal Cohen-Macaulay modules over a hypersurface $A$.
In this paper, we prove a noncommutative graded version of Kn\"orrer's periodicity theorem. 
Moreover, we prove another way to reduce the number of variables in computing 
the stable category 
$\uCM^{\ZZ}(A)$ 
of graded maximal Cohen-Macaulay modules 
if $A$ is a noncommutative quadric hypersurface. 
Under high rank property defined in this paper, we also show that computing $\uCM^{\ZZ}(A)$
over a noncommutative smooth quadric hypersurface $A$
in up to six variables can be reduced to one or two variables cases.
In addition, we give a complete classification of $\uCM^{\ZZ}(A)$
over a smooth quadric hypersurface $A$ in a skew $\mathbb P^{n-1}$, where $n \leq 6$, without high rank property using graphical methods.
\end{abstract} 

\maketitle

\section{Introduction} 
Throughout this section, we fix an algebraically closed field $k$ of characteristic not equal to $2$.  Let $S= k[[x_1,\dots,x_n]]$ be the formal power series ring of $n$ variables, and $f \in (x_1,\dots,x_n)^2 \subset S$ a nonzero element.
A matrix factorization of $f$ is a pair $(\Phi, \Psi)$ of $r \times r$ square matrices whose entries are elements in $S$ such that
\[ \Phi\Psi =\Psi\Phi =fI_r. \]
In \cite{Ei}, Eisenbud showed the factor category $\uMF_S(f):= \MF_S(f)/\add\{(1,f), (f,1)\}$ of the category $\MF_S(f)$ of matrix factorizations of $f$ is equivalent to the stable category $\uCM(S/(f))$ of maximal Cohen-Macaulay $S/(f)$-modules
(\cite[Section 6]{Ei}, see also \cite[Theorem 7.4]{Y}).
By this equivalence, we can apply the theory of (reduced) matrix factorizations to the representation theory of Cohen-Macaulay modules (with no free summand) over hypersurfaces.
In \cite{Kn}, Kn\"orrer proved the following famous theorem, which is now called Kn\"orrer's periodicity theorem:

\begin{theorem}[{\cite[Theorem 3.1]{Kn}, see also \cite[Theorem 12.10]{Y}}]
Let $S= k[[x_1,\dots,x_n]]$ and $0 \neq f \in (x_1,\dots,x_n)^2$. Then
\[ \uCM(S/(f)) \cong \uMF_{S}(f) \cong \uMF_{S[[u,v]]}(f+u^2+v^2) \cong \uCM(S[[u,v]]/(f+u^2+v^2)). \]
\end{theorem}

Kn\"orrer's periodicity theorem plays an essential role in representation theory of Cohen-Macaulay modules over Gorenstein rings.
For example, using Kn\"orrer's periodicity theorem,
we can show that arbitrary dimensional simple singularities are of finite Cohen-Macaulay representation type.
Furthermore, matrix factorizations are related to several areas of mathematics, including singularity categories, Calabi-Yau categories, Khovanov-Rozansky homology, and homological mirror symmetry.

In this paper, to study noncommutative hypersurfaces, which are important objects in noncommutative algebraic geometry,
we present some noncommutative graded versions of Kn\"orrer's periodicity theorem.
The notion of noncommutative (graded) matrix factorizations was introduced in \cite{MU}.
In terms of noncommutative graded matrix factorizations, the category $\uNMF^{\ZZ}_S(f)$ can be defined,
and a noncommutative graded analogue of Eisenbud's result was obtained as follows:

\begin{theorem}[{\cite[Theorem 6.6]{MU}}]
Let $S$ be a noetherian AS-regular algebra and $f\in S$ a regular normal homogeneous element.
Then the category $\uNMF^{\ZZ}_S(f)$ is equivalent to the stable category $\uCM^{\ZZ}(S/(f))$ of maximal graded Cohen-Macaulay $S/(f)$-modules.
\end{theorem}

This paper is organized as follows:  In Section 2, we collect some definitions and preliminary results needed in this paper. 

In Section 3, we prove the following result, which is a natural noncommutative graded analogue of Kn\"orrer's periodicity theorem:

\begin{theorem}[{Theorem \ref{thm.nkp}}] \label{thm.inkp} 
Let $S$ be a noetherian AS-regular algebra and $f\in S$ a regular normal homogeneous element of even degree.
If there exists a graded algebra automorphism $\s$ of $S$ such that 
$\s(f)=f$ and
$af=f\s^2(a)$ for every $a\in S$ (eg. if $f$ is central and $\s=\id$), then 
$$\uCM^{\ZZ}(S/(f)) \cong \uNMF_S^{\ZZ}(f) \cong \uNMF_{S[u; \s][v; \s]}^{\ZZ}(f+u^2+v^2) \cong \uCM^{\ZZ}(S[u; \s][v; \s]/(f+u^2+v^2))$$
where $S[u; \s][v; \s]$ is the Ore extension of $S$ by $\s$ with $\deg u=\deg v=\frac{1}{2}\deg f$. 
\end{theorem}

Since $f$ is regular normal, the category of noncommutative matrix factorizations of $f$ is equivalent to
the category of twisted matrix factorizations of $f$ by \cite[Proposition 4.7]{MU},
so the above theorem gives a categorical statement of \cite[Theorem 1.7]{CKMW}. A special case of the above theorem was recently proved in \cite [Theorem 8.2]{HY}. 

In Section 4, we focus on  
noncommutative quadric hypersurfaces. 
A homogeneous coordinate ring of a quadric hypersurface in a quantum $\PP^{n-1}$ is defined by $A=S/(f)$ where $S$ is an $n$-dimensional quantum polynomial algebra and $f\in S_2$ is a regular normal element \cite {SV}.   Kn\"orrer's periodicity theorem is a powerful tool to compute $\uCM^{\ZZ}(A)$ since it reduces the number of variables.  If $f$ is a central element, then there is another way to reduce the number of variables, which applies only in the noncommutative setting.   

\begin{theorem}[{Theorem \ref{thm.ca2}}]
\label{thm.ttpr} 
If $A=S/(f)$ is a homogeneous coordinate ring of a quadric hypersurface in a quantum $\PP^{n-1}$ where $f\in S_2$ is a regular central element, 
then 
$$\uCM^{\ZZ}(S[u; -1][v;-1]/(f+u^2+v^2))\cong \uCM^{\ZZ}(S[u;-1]/(f+u^2))\times \uCM^{\ZZ}(S[v;-1]/(f+v^2))$$
where $-1$ denotes a graded algebra automorphism defined by $a\mapsto (-1)^{\deg a}a$ for a homogeneous element $a$.    
\end{theorem} 

In Section 5, we focus on noncommutative smooth quadric hypersurfaces.
It is well-known that $A$ is the homogeneous coordinate ring of a (commutative) smooth quadric hypersurface in $\PP^{n-1}$ if and only if
$A\cong k[x_1, \dots, x_n]/(x_1^2+\cdots +x_n^2)$.  Applying the graded Kn\"orrer's periodicity theorem to the commutative case, we have 
$$\uCM^{\ZZ}(A)\cong \begin{cases} \uCM^{\ZZ}(k[x_1]/(x_1^2)) & \text { if $n$ is odd,} \\
\uCM^{\ZZ}(k[x_1, x_2]/(x_1^2+x_2^2)) & \text { if $n$ is even.}\end{cases}$$
We prove a noncommutative analogue of this result up to 6 variables under  high rank property defined in this paper, which is an extension of the notion of irreducibility of $f$.  

\begin{theorem}[{Theorem \ref{thm.sKn1}}]\label{thm.isKn1}
Let $A=S/(f)$ be a homogeneous coordinate ring of a smooth high rank quadric hypersurface in a quantum $\PP^{n-1}$, where $n \leq 6$. Then
$$\uCM^{\ZZ}(A)\cong \begin{cases} \uCM^{\ZZ}(k[x_1]/(x_1^2)) & \text { if $n$ is odd,} \\
\uCM^{\ZZ}(k[x_1, x_2]/(x_1^2+x_2^2)) & \text { if $n$ is even.}\end{cases}$$
\end{theorem}

In Section 6, we show that the above theorem fails without ``high rank property'' by considering a simple example $A=A_\e := S_\e/(f_e)$
where $S_\e =k\<x_1, \dots, x_n\>/(x_ix_j-\e_{ij}x_jx_i)$ is a skew polynomial algebra and $f_\e = x_1^2+\cdots +x_n^2 \in S_\e$.  We use graphical methods.  To do this, we associate to each skew polynomial algebra $S_{\e}$ a certain graph $G_{\e}$.  We introduce in this paper four graphical operations, called mutation, relative mutation, Kn\"orrer reduction, and two point reduction for $G_{\e}$, and show that they are very powerful in computing $\uCM^{\ZZ}(A_{\e})$ as in the following lemma (Kn\"orrer reduction is a consequence of Theorem \ref{thm.inkp} and two point reduction is a consequence of Theorem \ref{thm.ttpr}): 

\begin{lemma} Let $G_{\e}, G_{\e'}$ be graphs associated to skew polynomial algebras $S_{\e}, S_{\e'}$.  
\begin{enumerate}
\item{} If $G_{\e'}$ is obtained from $G_{\e}$ by mutation, then $\uCM^{\ZZ}(A_{\e})\cong \uCM^{\ZZ}(A_{\e'})$ (Lemma \ref{lem.ramu}). 
\item{} If $G_{\e'}$ is obtained from $G_{\e}$ by relative mutation, then $\uCM^{\ZZ}(A_{\e})\cong \uCM^{\ZZ}(A_{\e'})$ (Lemma \ref{lem.M2}). 
\item{} If $G_{\e'}$ is obtained from $G_{\e}$ by Kn\"orrer reduction, then $\uCM^{\ZZ}(A_{\e})\cong \uCM^{\ZZ}(A_{\e'})$ (Lemma \ref{lem.R2}). 
\item{} If $G_{\e'}$ is obtained from $G_{\e}$ by two point reduction, then $\uCM^{\ZZ}(A_{\e})\cong \uCM^{\ZZ}(A_{\e'})\times \uCM^{\ZZ}(A_{\e'})$ (Lemma \ref{lem.R1}). 
\end{enumerate}
\end{lemma} 

By using these four graphical operations, we will give a complete classification of $\uCM^{\ZZ}(A_\e)$ up to $n=6$ (Section \ref{ss.class}).

\section{Preliminaries} 

\subsection{Terminology and Notation}

Throughout this paper, we fix a field $k$. 
Unless otherwise stated, an \emph{algebra} means an algebra over $k$, and a \emph{graded ring} means a $\ZZ$-graded ring.

For a ring $A$, we denote by $\Mod A$ the category of right $A$-modules,
and by $\mod A$ the full subcategory consisting of finitely generated modules.
We denote by $A^o$ the opposite ring of $A$.

For a graded ring $A =\bigoplus_{i\in\ZZ} A_i$,
we denote by $\GrMod A$ the category of graded right $A$-modules,
and by $\grmod A$ the full subcategory consisting of finitely generated modules.
Morphisms in $\GrMod A$ are right $A$-module homomorphisms preserving degrees.
For $M \in \GrMod A$ and $n \in \ZZ$, we define the truncation $M_{\geq n} := \bigoplus_{i\geq n} M_i$ and the shift $M(n) \in \GrMod A$, which has the same underlying module structure as $M$, but which satisfies $M(n)_i = M_{n+i}$. 

Let $\cC$ be an additive category and $\cP$ a set of objects of $\cC$ closed under direct sums.
Then the \emph{factor category} $\cC/\cP$ has $\Obj(\cC/\cP) =\Obj(\cC)$ and $\Hom_{\cC/\cP}(M, N) = \Hom_{\cC}(M, N)/\cP(M,N)$ for $M, N\in \Obj(\cC/\cP)=\Obj(\cC)$,
where $\cP(M,N)$ is the subgroup consisting of all morphisms from $M$ to $N$ that factor through objects in $\cP$. Note that $\cC/\cP$ is also an additive category.

For a ring $A$, the \emph{stable category} of $\mod A$ is defined by $\umod A:= \mod A/\cP$ where $\cP:=\{P\in \mod A\mid P \; \textnormal{is projective}\}$.
The set of homomorphisms in $\umod A$ is denoted by $\uHom_A(M, N):=\Hom_A(M, N)/\cP(M, N)$ for $M, N\in \mod A$.
(For a graded ring $A$,  the \emph{stable category} of $\grmod A$ is defined by $\ugrmod A:= \grmod A/\cP$ where $\cP:=\{P\in \grmod A \mid P \; \textnormal{is projective}\}$.)

Let $A$ be a (graded) ring. A (graded) right $A$-module $M$ is called \emph{totally reflexive} if
\begin{enumerate}
\item{} $\Ext^i_A(M, A)=0$ for all $i\geq 1$, 
\item{} $\Ext^i_{A^o}(\Hom_A(M, A), A)=0$ for all $i\geq 1$, and 
\item{} the natural biduality map $M\to \Hom_{A^o}(\Hom_A(M, A), A)$ is an isomorphism.
\end{enumerate}
The full subcategory of $\mod A$ consisting of totally reflexive modules is denoted by $\TR(A)$.
(The full subcategory of $\grmod A$ consisting of graded totally reflexive modules is denoted by $\TR^{\ZZ}(A)$.)
The stable category of $\TR(A)$ is defined by $\uTR(A):= \TR(A)/\cP$. 
(The stable category of $\TR^{\ZZ}(A)$ is defined by $\uTR^{\ZZ}(A):= \TR^{\ZZ}(A)/\cP$.)

Let $A$ be a ring and $M\in \Mod A$. In this paper, we use the notation $\Omega M$ in the following sense:  we first choose a projective resolution of $M$ 
\[
\xymatrix@R=0.5pc@C=3pc{
\cdots \ar[r] &F^2 \ar[r]^{\phi^1} &F^1 \ar[r]^{\phi^0} &F^0 \ar[r]^{\e_M} &M \ar[r] &0 
}
\] and define $\Omega M=\Ker \e_M$.  It means that $\Omega M$ depends on the choice of a  projective resolution of $M$, however, such a dependency will be gone in the stable category.  In our main applications, we will always choose a minimal free resolution of $M$ to define $\Omega M$ so that it is uniquely determined up to isomorphism.

Let $A$ be a ring and $M, N\in \Mod A$.  Then the connecting homomorphism $\Hom_A(\Omega M, N)\to \Ext_A^1(M, N)$ is given by  
$h\mapsto \left\{0\to N\to \Coker \begin{pmatrix} \psi^0 & \mu^0 \\ 0 & \phi^0 \end{pmatrix} \to M\to 0\right\}$
where 
\[
\xymatrix@R=0.5pc@C=3pc{
F^2 \ar[r]^{\phi^1} &F^1 \ar[r]^{\phi^0} &F^0 \ar[r]^{\e_M} &M \ar[r] &0 \\
G^2 \ar[r]^{\psi^1} &G^1 \ar[r]^{\psi^0} &G^0 \ar[r]^{\e_N} &N \ar[r] &0
}
\] 
are projective resolutions of $M, N$, and $\mu^i:F^{i+1}\to G^i$ is a lift of 
$h\in \Hom_A(\Omega M, N)$, that is, 
\[\xymatrix@R=2pc@C=3pc{
F^{2} \ar[d]_{\mu^1} \ar[r]^{\phi^1} &F^{1} \ar[d]_{\mu^0} \ar[r]^{\e_{\Omega M}} &\Omega M \ar[d]_{h} \ar[r]^{} &0 \\
G^{1} \ar[r]^{\psi^0} &G^{0} \ar[r]^{\e_{N}} &N \ar[r]^{} &0
}\]
is a commutative diagram. 

\begin{lemma} \label{lem.ch} Let $A$ be a ring and $M, N\in \TR(A)$.
Then the map $\delta:\Hom_{A}(M, N)\to \Ext^1_{A}(M, \Omega N)$ defined by 
$$\delta(h)=\left\{0\to \Omega N\to \Coker \begin{pmatrix} \psi^1 & \mu^1 \\ 0 & \phi^0 \end{pmatrix} \to M\to 0\right\}$$ 
induces a bijection $\underline{\delta}:\uHom_{A}(M, N)\to \Ext^1_{A}(M, \Omega N)$ 
where 
\[
\xymatrix@R=0.5pc@C=3pc{
F^2 \ar[r]^{\phi^1} &F^1 \ar[r]^{\phi^0} &F^0 \ar[r]^{\e_M} &M \ar[r] &0 \\
G^2 \ar[r]^{\psi^1} &G^1 \ar[r]^{\psi^0} &G^0 \ar[r]^{\e_N} &N \ar[r] &0
}
\] 
are projective resolutions of $M, N$, and $\mu^i:F^{i}\to G^i$ is a lift of 
$h\in \Hom_A(M, N)$.  A graded version of the above statement also holds.  
\end{lemma} 

\begin{proof}
Since $M, N\in \TR(A)$, the map $\Hom_A(M, N)\to \Hom_A(\Omega M, \Omega N);
\; h\mapsto \Omega h$ induces a bijection $\uHom_A(M, N)\to \uHom_A(\Omega M, \Omega N)$.  The lift of $\Omega h$ is given by 
\[\xymatrix@R=2pc@C=3pc{
F^{2} \ar[d]_{\mu^2} \ar[r]^{\phi^1} &F^{1} \ar[d]_{\mu^1} \ar[r]^{\e_{\Omega M}} &\Omega M \ar[d]_{\Omega h} \ar[r]^{} &0 \\
G^{2} \ar[r]^{\psi^1} &G^{1} \ar[r]^{\e_{\Omega N }} &\Omega N  \ar[r]^{} &0,
}\]
so the connecting homomorphism  
$\Hom_A(\Omega M, \Omega N)\to \Ext^1_A(M, \Omega N)$ is given by 
$$\Omega h\mapsto \left\{0\to \Omega N\to
\Coker \begin{pmatrix} \psi^1 & \mu^1 \\ 0 & \phi^0 \end{pmatrix} \to M\to 0 \right\}.$$ 
Since $M\in \TR(A)$, it induces a bijection $\uHom_A(\Omega M, \Omega N)\to \Ext^1_A(M, \Omega N)$,
so we have the result. 
\end{proof}   

For a graded algebra $A =\bigoplus_{i\in\ZZ} A_i$, we say that $A$ is \emph{connected graded} if $A_i=0$ for all $i<0$ and $A_0 = k$,
and we say that $A$ is \emph{locally finite} if $\dim_k A_i <\infty$ for all $i \in \ZZ$.
If $A$ is a locally finite graded algebra and $M \in \grmod A$, then we define the \emph{Hilbert series} of $M$ by
$H_M (t) := \sum_{i \in \ZZ} (\dim_k M_i)t^i \in \ZZ[[t, t^{-1}]]$.

We recall a nice operation for graded algebras, called twisting system, introduced by Zhang \cite{Zh}.  
Let $A$ be a graded algebra. A \emph{twisting system} on $A$ is a sequence $\theta=\{\theta_i\}_{i\in \ZZ}$ of graded $k$-linear automorphisms of $A$ such that $\theta_i(a\theta_j(b))=\theta_i(a)\theta_{i+j}(b)$ for every $i, j\in \ZZ$ and every $a\in A_j, b\in A$.
The twisted graded algebra of $A$ by a twisting system $\theta $ is a graded algebra $A^{\theta}$ where $A^{\theta}=A$ as a graded $k$-vector space with the new multiplication $a^{\theta}b^{\theta}=(a\theta_{i}(b))^{\theta}$ for $a^{\theta}\in A^{\theta}_i, b^{\theta}\in A^{\theta}$.
Here we write $a^{\theta}\in A^{\theta}$ for $a\in A$ when viewed as an element of $A^{\theta}$ and the product $a\theta_{i}(b)$ is computed in $A$.
We denote by $\GrAut A$ the group of graded $k$-algebra automorphisms of $A$.  If $\theta\in \GrAut A$, then $\{\theta^i\}_{i\in \ZZ}$ is a twisting system of $A$.  In this case, we simply write $A^{\theta}:=A^{\{\theta^i\}}$.

\begin{lemma}  [{\cite {Zh}}] \label{lem.ztw} If $A$ is a graded algebra and $\theta$ is a twisting system on $A$, then $\GrMod A^{\theta}\cong \GrMod A$.  
\end{lemma}   

The following classes of algebras are main objects of study in noncommutative algebraic geometry.  

\begin{definition} 
A connected graded algebra $A$ is called an \emph{AS-regular} (resp. \emph{AS-Gorenstein}) algebra of dimension $n$ if
\begin{enumerate}
\item{} $\gldim A =n <\infty$ (resp. $\injdim_A A = \injdim_{A^o} A= n <\infty$), and
\item{} $\Ext^i_A(k ,A) \cong \Ext^i_{A^o}(k ,A) \cong
\begin{cases}
0 & \textnormal { if }\; i\neq n\\
k (\ell) \; \textrm{for some}\; \ell \in \ZZ 
& \textnormal { if }\; i=n
\end{cases}$
where $k:=A/A_{\geq 1}\in \GrMod A$.  
\end{enumerate}
A \emph{quantum polynomial algebra} of dimension $n$ is a noetherian AS-regular algebra $A$ of dimension $n$ with $H_A(t)=(1-t)^{-n}$.  
\end{definition}

A quantum polynomial algebra of dimension $n$ is a noncommutative analogue of the commutative polynomial algebra in $n$ variables of degree 1.

\begin{lemma} \label{lem.bal} 
Let $S$ be a connected graded algebra and $f\in S_d$ a regular normal element of positive degree.
For $n\geq 1$, $S$ is a (noetherian) AS-Gorenstein algebra of dimension $n$ if and only if $S/(f)$ is a (noetherian) AS-Gorenstein algebra of dimension $n-1$. 
\end{lemma} 

\begin{proof}  
This follows by Rees' Lemma (eg. \cite[Proposition 3.4(b)]{L}).
\end{proof} 

If $A$ is a noetherian connected graded algebra, then the full subcategory of $\GrMod A$ consisting of direct limits of finite dimensional modules over $k$ is denoted by $\Tors A$,
the quotient category $\GrMod A/\Tors A$ is denoted by $\Tails A$, and the quotient functor is denoted by $\pi :\GrMod A\to \Tails A$.
We also write $\tors A:=\Tors A\cap \grmod A$ and $\tails A:=\grmod A/\tors A$.
We call $\tails A$ the \emph{noncommutative projective scheme} associated to $A$,
and $A$ a \emph{homogeneous coordinate ring} of $\tails A$.
Note that $\Tails A$ and $\tails A$ are also written as $\QGr A$ and $\qgr A$, respectively, by some authors.

If $A$ is a quantum polynomial algebra of dimension $n$, then we call $\tails A$ a \emph{quantum $\PP^{n-1}$}.  

Let $A$ be a noetherian AS-Gorenstein algebra of dimension $n$. 
We define the \emph{$i$-th local cohomology} of $M \in \grmod A$ by
$\H^i_\fm(M):= \lim _{n \to \infty} \Ext^i_A(A/A_{\geq n}, M)$.
Then one can show that $\H_{\fm}^i(A)=0$ for all $i\neq n$.
A graded module $M\in \grmod A$ is called \emph{maximal Cohen-Macaulay} if $\H_{\fm}^i(M)=0$ for all $i\neq n$.
By \cite[Lemma 4.6]{Mbc}, $M\in \grmod A$ is maximal Cohen-Macaulay if and only if it is totally reflexive,
so in this setting, we also use notation $\CM^{\ZZ}(A)$ and $\uCM^{\ZZ}(A)$ for $\TR^{\ZZ}(A)$ and $\uTR^{\ZZ}(A)$, respectively.

Let $A$ be a connected graded algebra.  We say that $M\in \GrMod A$ has a \emph{linear resolution} if $M$ has a free resolution $\cdots \to F^2\to F^1\to F^0\to M\to 0$ such that each $F^i$ is generated in degree $i$.  The full subcategory of $\grmod A$ consisting of modules having linear resolutions is denoted by $\lin A$.  We say that $A$ is \emph{Koszul} if $k:=A/A_{\geq 1}\in \lin A$.  Note that if $A$ is a Koszul algebra, then $A^!:=\bigoplus _{i\in \NN}\Ext^i_A(k, k)$ is also a Koszul algebra, called the \emph{Koszul dual} algebra.  

\begin{lemma} 
\label{lem.Koz}
Suppose that $A$ and $A^!$ are both noetherian Koszul AS-Gorenstein algebras.
\begin{enumerate}
\item{} The Koszul duality $E:\lin A\to \lin (A^!)^o$ defined by $E(M)=\bigoplus _{i\in \NN}\Ext^i_A(M, k)$ extends to a duality $\overline E:\cD^b(\grmod A)\to \cD^b(\grmod (A^!)^o)$ such that ${\overline E}(M[p](q))={\overline E}(M)[-p-q](q)$. 
\item{}
$\overline{E}$ induces a duality $B:\uCM^{\ZZ}(A)\to \cD^b(\tails (A^!)^o)$, which induces a duality $\uCM^{\ZZ}(A)\cap \lin A \to \tails (A^!)^o$.
\end{enumerate}
\end{lemma} 

\begin{proof}
(1) This follows from \cite [Proposition 4.5]{Mrp}.  

(2) This follows from \cite [Theorem 5.3]{Mrr} and \cite [Theorem 3.2]{SV}. 
\end{proof}

\subsection{Noncommutative Matrix Factorizations} 

In this subsection, we recall some background results on noncommutative matrix factorizations obtained in \cite{MU}. 

\begin{definition}[{\cite[Definition 2.1]{MU}}]
Let $S$ be a ring and $f\in S$ an element.
A \emph{noncommutative right matrix factorization} of $f$ over $S$ of rank $r$ is a sequence of right $S$-module homomorphisms $\{\phi^i:F^{i+1}\to F^i\}_{i\in \ZZ}$
where $F^i$ are free right $S$-modules of rank $r$ for some $r\in \NN$ such that there is a commutative diagram
\[\xymatrix@R=2pc@C=3pc{
F^{i+2} \ar[d]_{\cong} \ar[r]^{\phi^i\phi^{i+1}} &F^i \ar[d]^{\cong} \\
S^r \ar[r]^{f\cdot} &S^r 
}\]
for every $i\in \ZZ$.
A morphism $\mu :\{\phi^i:F^{i+1}\to F^i\}_{i\in \ZZ}\to \{\psi^i:G^{i+1}\to G^i\}_{i\in \ZZ}$
of noncommutative right matrix factorizations is a sequence of right $S$-module homomorphisms $\{\mu ^i:F^i\to G^i\}_{i\in \ZZ}$ such that the diagram 
\[\xymatrix@R=2pc@C=3pc{
F^{i+1} \ar[d]_{\mu ^{i+1}} \ar[r]^{\phi^i} &F^i \ar[d]^{\mu ^{i}} \\
G^{i+1} \ar[r]^{\psi^i} &G^{i}
}\]
commutes for every $i\in \ZZ$.
We denote by $\NMF_S(f)$ the category of noncommutative right matrix factorizations.  

Let $S$ be a graded ring and $f\in S_d$ a homogeneous element.
A \emph{noncommutative graded right matrix factorization} of $f$ over $S$ of rank $r$ is a sequence of graded right $S$-module homomorphisms $\{\phi^i:F^{i+1}\to F^i\}_{i\in \ZZ}$
where $F^i$ are graded free right $S$-modules of rank $r$ for some $r\in \NN$ such that 
there is a commutative diagram 
\[\xymatrix@R=2pc@C=0.75pc{
F^{i+2} \ar[d]_{\cong} \ar[rrr]^{\phi^i\phi^{i+1}} &&&F^i \ar[d]^{\cong} \\
\bigoplus _{s=1}^rS(-m_{i+2,s})\ar@{=}[r] &\bigoplus _{s=1}^rS(-m_{is}-d) \ar[rr]^(0.54){f\cdot} &&\bigoplus _{s=1}^rS(-m_{is})  
}\]
for some $m_{is} \in \ZZ$ and every $i\in \ZZ$. 
We can similarly define the category of noncommutative graded right matrix factorizations $\NMF^{\ZZ}_S(f)$.  
\end{definition} 

\begin{remark} \label{rem.lambda}
Let $S$ be a (graded) ring and $f\in S$ a (homogeneous) element. 
\begin{enumerate}
\item{}  Let $\{\phi^i:F^{i+1}\to F^i\}_{i\in \ZZ}$ be a noncommutative right matrix factorization of $f$ over $S$ of rank $r$.
We often assume without loss of generality that $F^i=S^r$ and $\phi^i\phi^{i+1}=f\cdot$ (see \cite[Remark 2.2 (1)]{MU}).  In this case, every $\phi^i$ is the left multiplication of a matrix $\Phi^i$ whose entries are elements in $S$, so that $\Phi^i\Phi^{i+1}=fI_r$ where $I_r$ is the identity matrix of size $r$.  
\item{}  Let $\{\phi^i:F^{i+1}\to F^i\}_{i\in \ZZ}$ be a noncommutative graded right matrix factorization of $f$ over $S$ of rank $r$ such that $F^i=\bigoplus _{s=1}^rS(-m_{is})$.
In this case, we may write $\phi^i=(\phi^i_{st})$ where $\phi^i_{st}:S(-m_{i+1, t})\to S(-m_{is})$ is the left multiplication of an element in $S_{m_{i+1, t}-m_{is}}$, so $\phi^i$ is the left multiplication of a matrix $\Phi^i$ whose entries are homogeneous elements in $S$, so that $\Phi^i\Phi^{i+1}=fI_r$ where $I_r$ is the identity matrix of size $r$.  
\end{enumerate}
\end{remark}

The following lemma is immediate.  

\begin{lemma} \label{lem.eq1}
If $\varphi:S\to S'$ is an isomorphism of rings and $f\in S$, then $\NMF_S(f)\cong \NMF_{S'}(\varphi(f))$. 
\end{lemma} 

For an algebra $S$ and an element $f\in S$, we define $\Aut (S; f):=\{\s\in \Aut S\mid \s(f)=\lambda f\ \textnormal{for some}\ \lambda\in k\}$, and $\Aut _0(S; f):=\{\s\in \Aut S\mid \s(f)=f\}\leq \Aut (S; f)$.
For a graded algebra $S$ and a homogeneous element $f\in S$, we define $\GrAut (S; f)$ and $\GrAut_0(S; f)$ similarly.

\begin{remark} There exists a canonical map $\Aut (S; f)\to \Aut (S/(f))$, which is not injective nor surjective in general.  For example, it is easy to see that $\Aut (k[x]; x)\to \Aut (k[x]/(x))$ is not injective, and $\Aut (k[x, y]/(x^2); y^2)\to \Aut (k[x, y]/(x^2, y^2))$ is not surjective.  
\end{remark}  

\begin{lemma}[{\cite [Theorem 3.7]{MU}}] \label{lem.eq2} 
Let $S$ be a graded algebra and $f\in S$ a homogeneous element.
For $\theta\in \GrAut (S; f)$, we have $\NMF_S^{\ZZ}(f)\cong \NMF_{S^{\theta}}^{\ZZ}(f^{\theta})$.  
\end{lemma} 

Let $S$ be a (graded) ring and $f\in S$ a (homogeneous) regular normal element. 
Then there exists a unique (graded) ring automorphism $\nu_f$ of $S$ such that $af=f\nu_f(a)$ for $a\in S$.  
We call $\nu_f$ the \emph{normalizing automorphism} of $f$.  

Let $\s$ be a (graded) ring automorphism of $S$.
If $\Phi=(a_{st})$ is a matrix whose entries are (homogeneous) elements in $S$,
then we write $\s(\Phi)=(\s(a_{st}))$.
If $\phi$ is a (graded) right $S$-module homomorphism given by the left multiplication of $\Phi$,
then we write $\s(\phi)$ for the (graded) right $S$-module homomorphism given by the left multiplication of $\s(\Phi)$.  

\begin{theorem}[{\cite[Theorem 4.4 (1), (3)]{MU}}] \label{thm.nmf}
Let $S$ be a (graded) ring and $f\in S$ a (homogeneous) regular normal element (of degree $d$). 
\begin{enumerate}
\item{} If $\phi$ is a noncommutative (graded) right matrix factorization of $f$ over $S$, then $\phi^{i+2}=\nu_f(\phi^i)$ ($\phi^{i+2}=\nu_f(\phi^i)(-d)$) for every $i\in \ZZ$.  It follows that $\phi$ is uniquely determined by $\phi^0$ and $\phi^1$. 
\item{} If $\mu:\phi\to \psi$ is a morphism of noncommutative (graded) right matrix factorizations of $f$ over $S$, then $\mu^{i+2}=\nu_f(\mu^i)$ ($\mu^{i+2}=\nu_f(\mu^i)(-d)$) for every $i\in \ZZ$.  It follows that $\mu$ is uniquely determined by $\mu^0$ and $\mu^1$. 
\end{enumerate}
\end{theorem}

Let $S$ be a (graded) ring, $f\in S$ a (homogeneous) regular normal element,
and $A=S/(f)$.
For $a\in S$, write $\overline a\in A$, and for $\phi:F\to G$ where $F, G$ are (graded) free $S$-modules, write $\overline\phi:\overline F\to \overline G$. 
For a noncommutative (graded) right matrix factorization $\phi$ of $f$ over $S$, we define the complex $C (\phi)$ of (graded) right $A$-modules by 
\[
\xymatrix@C=3pc@R=1pc{
\cdots \ar[r]^{\overline {\phi^2}}
&\overline {F^2} \ar[r]^{\overline {\phi^1}}
&\overline {F^1} \ar[r]^{\overline {\phi^0}}
&\overline {F^0} \ar[r]^{\overline {\phi^{-1}}}
&\overline {F^{-1}}  \ar[r]^{\overline {\phi^{-2}}} 
&\cdots.}
\]
Moreover we define $\Coker \phi:=\overline{\Coker \phi^0} \cong \Coker \overline{\phi^0}\in \mod A$ ($\grmod A$).

\begin{definition}[{\cite[Definition 6.3]{MU}}] \label{def.trivnmf}
Let $S$ be a ring and $f\in S$.  For a free right $S$-module $F$, we define $\phi_F, {_F\phi}\in \NMF_S(f)$ by
$$\begin{array}{lll}
& \phi_F^{2i}=\id_F:F\to F, & \phi_F^{2i+1}=f\cdot : F\to F, \\
& {_F\phi}^{2i}=f\cdot:F\to F, & {_F\phi}^{2i+1}=\id _F: F\to F.
\end{array}$$
We define 
$\cF :=\{\phi_F\mid F\in \mod S \; \textnormal{is free} \}$,
$\cG :=\{\phi_F\oplus {_G\phi} \mid F, G\in \mod S \; \textnormal{are free} \}$ and
$\uNMF_S(f):=\NMF_S(f)/\cG$.

Let $S$ be a graded ring and $f\in S_d$.  For a graded free right $S$-module $F$, 
we define $\phi_F, {_F\phi}\in \NMF_S^{\ZZ}(f)$ by
$$\begin{array}{lll}
& \phi_F^{2i}=\id_F:F(-id)\to F(-id), & \phi_F^{2i+1}=f\cdot : F(-id-d)\to F{(-id)}, \\
& {_F\phi}^{2i}=f\cdot: F(-id-d)\to F(-id), & {_F\phi}^{2i+1}=\id _F: F(-id-d)\to F(-id-d).
\end{array}$$
We define
$\cF :=\{\phi_F\mid F\in \grmod S \; \textnormal{is free} \}$,
$\cG :=\{\phi_F\oplus {_G\phi} \mid F, G\in \grmod S \; \textnormal{are free} \}$ and
$\uNMF_S^{\ZZ}(f):=\NMF_S^{\ZZ}(f)/\cG$.
\end{definition}

\begin{theorem}[{\cite[Theorem 6.6]{MU}}]  \label{thm.m4} 
If $S$ is a noetherian AS-regular algebra, $f\in S_d$ is a regular normal element, and $A=S/(f)$, 
then the functor $\Coker:\NMF^{\ZZ}_S(f)\to \TR^{\ZZ}(A)=\CM^{\ZZ}(A)$ induces an equivalence functor
$\underline {\Coker}:\uNMF^{\ZZ}_S(f)\to \uTR^{\ZZ}(A) = \uCM^{\ZZ}(A)$.
\end{theorem}

\section{Noncommutative Kn\"orrer's Periodicity Theorem} 

In this section, we prove a noncommutative graded version of Kn\"orrer's periodicity theorem following the methods in \cite{Y}.  In the noncommutative setting, we use an Ore extension in place of the polynomial extension following \cite {CKMW}.

\subsection{Ore Extensions} 

We first list some basic properties of an Ore extension which are needed in this paper.  Let $S$ be a ring and $\s\in \Aut S$.  An \emph{Ore extension} $S[u; \s]$ of $S$ by $\s$ is $S[u; \s]=S[u]$ as a free right $S$-module such that $au=u\s(a)$ for $a\in S$.  It is easy to see that $u\in S[u; \s]$ is a regular normal element with the normalizing automorphism $\nu_u\in \Aut (S[u; \s])$ uniquely determined by $\nu_u|_S=\s$ and $\nu _u(u)=u$.  

\begin{lemma} Let $S$ be a ring and $\s, \t\in \Aut S$.  Then there exists $\bar \t\in \Aut S[u; \s]$ such that $\bar \t|_S=\t$ and $\bar \t(u)=u$ if and only if $\s\t=\t\s$. 
\end{lemma} 

\begin{proof} Since $u$ is a regular element, $\bar \t(au-u\s(a))=\t(a)u-u\t(\s(a))=u(\s(\t(a))-\s(\t(a)))=0$ for every $a\in S$ if and only if $\s\t=\t\s$. 
\end{proof}

By abuse of notation, we write $\bar \t=\t$ in the sequel.  

\begin{lemma} \label{lem.rn} 
Let $f\in S$ be a regular normal element, $\s, \t\in \Aut _0(S; f)$. 
If $\s\t=\t\s=\nu_f$, then $f+uv\in S[u; \s][v; \t]$ is a regular normal element. 
\end{lemma} 

\begin{proof} If $\s\t=\nu_f$, then $a(f+uv)=f\nu_f(a)+uv\t\s(a)=(f+uv)\nu_f(a)$ for every $a\in S$, $u(f+uv)=\s^{-1}(f)u+uvu=(f+uv)u$, and $v(f+uv)=\t^{-1}(f)v+uv^2=(f+uv)v$, so $f+uv\in S[u; \s][v; \t]$ is a normal element.

Define the degree on $S[u; \s][v; \t]$ by $\deg (S)=0$ and $\deg u=\deg v=1$.  For $0\neq a=\sum _{i=0}^da_i\in S[u; \s][v; \t]$ where $\deg a_i=i$ such that $a_d\neq 0$, if $a(f+uv)=0$, then $a_duv=0$.  Since $uv$ is a regular element in $S[u; \s][v; \t]$, we have $a_d=0$, which is a contradiction, so $f+uv$ is regular.  
\end{proof} 

\begin{example} \label{ex.rn} 
Let $S$ be a ring and $f\in S$ a regular normal element.  The following pair $(\s, \t)$ of automorphisms of $S$ satisfies the conditions $\s, \t\in \Aut _0(S; f)$ 
and $\s\t=\t\s=\nu_f$.
\begin{enumerate}
\item{} $(\s, \t)=(\sqrt {\nu_f}, \sqrt{\nu_f})$ if $\sqrt {\nu_f}\in \Aut _0(S; f)$
exists (see \cite {CKMW}). 
\item{} $(\s, \t)=(\nu_f, \id), (\id, \nu_f)$, which always exist. 
\end{enumerate}
\end{example} 

\begin{remark} \label{rem.uvf} 
Let $f\in S$ be a regular normal element, $\s, \t\in \Aut _0(S; f)$ such that 
$\s\t=\t\s=\nu_f$.  By the proof of Lemma \ref{lem.rn}, $\nu_u, \nu_v, \nu_{f+uv}\in \Aut (S[u; \s][v;\t])$ are uniquely determined by 
\begin{align*}
& \nu_u|_S=\s,\; \nu_u(u)=u,\; \nu_u(v)=v, \\
& \nu_v|_S=\t,\; \nu_v(u)=u,\; \nu_v(v)=v, \\
& \nu_{f+uv}|_S=\nu_f,\; \nu_{f+uv}(u)=u,\; \nu_{f+uv}(v)=v.   
\end{align*}
\end{remark} 

Note that if $S$ is a graded ring and $\s\in \GrAut S$, then $S[u; \s]$ is naturally a graded ring.  In this subsection, we stated and proved results in the ungraded setting, but graded versions of these results also hold.  

\subsection{Noncommutative Kn\"orrer's Periodicity Theorem} 

In this subsection, we state and prove results in the graded setting although ungraded versions of those results also hold except for Theorem \ref{thm.nkp}.   
{\bf In the sequel, we tacitly assume that $S$ is an $\NN$-graded ring, and $\deg f, \deg u, \deg v>0$ such that $\deg u+\deg v=\deg f$, so that $f+uv\in S[u; \s][v; \t]$ is a homogeneous element.}

Let $S$ be a graded ring, and $\s, \t\in \GrAut S$.  A homomorphism $\phi:M\to N$ in $\GrMod S$ induces a homomorphism $\phi\otimes \id:M\otimes _SS[u; \s][v; \t]\to N\otimes _SS[u; \s][v; \t]$ in $\GrMod S[u; \s][v;\t]$, which we denote by $\phi:M[u; \s][v; \t]\to N[u; \s][v; \t]$ by abuse of notation since both maps are given by the left multiplication of the same matrix.  

\begin{lemma} Let $S$ be a graded ring, $f\in S_d$ a regular normal element, and $\s, \t\in \GrAut _0(S; f)$.  
If 
$\s\t=\t\s=\nu_f$, then
$\cH :\NMF_S^{\ZZ}(f)\to \NMF_{S[u; \s][v; \t]}^{\ZZ}(f+uv)$ defined by 
\begin{align*}
& \cH(\{\phi^i:F^{i+1}\to F^i\}_{i\in \ZZ})= \\
& \left\{\begin{pmatrix} \t(\phi^i) & u\cdot  \\ v\cdot  & -\phi^{i+1}(e) \end{pmatrix}:F^{i+1}[u; \s][v; \t]\oplus F^{i+2}[u; \s][v; \t](e) \to F^i[u; \s][v; \t]\oplus F^{i+1}[u; \s][v; \t](e)\right\}_{i\in \ZZ}, \\
& \cH(\{\mu^i\}_{i\in \ZZ})=\left\{\begin{pmatrix} \t(\mu^{i}) & 0 \\ 0 & \mu^{i+1}(e) \end{pmatrix}\right\}_{i\in \ZZ}
\end{align*} 
is a functor,  which induces a functor $\underline \cH :\uNMF_S^{\ZZ}(f)\to \uNMF_{S[u; \s][v; \t]}^{\ZZ}(f+uv)$ where $\deg v=e$ so that $\deg u=d-e$, and $u\cdot :F^{i+2}[u; \s][v; \t](e)\cong F^i[u; \s][v; \t](e-d)\to F^i[u; \s][v; \t]$. 
\end{lemma} 

\begin{proof} 
For $\phi\in \NMF_S^{\ZZ}(f)$, since $\phi^{i+2}(e)=\nu_f(\phi^i)(e-d)$
for every $i\in \ZZ$ by 
Theorem \ref{thm.nmf} (1), 
\begin{align*} 
& \begin{pmatrix} \t(\phi^i) & u\cdot \\ v\cdot & -\phi^{i+1}(e) \end{pmatrix}\begin{pmatrix} \t(\phi^{i+1}) & u\cdot \\ v\cdot & -\phi^{i+2}(e) \end{pmatrix}=\begin{pmatrix} \t(\phi^i)\t(\phi^{i+1})+uv & \t(\phi^i)u-u\phi^{i+2}(e) \\ v\t(\phi^{i+1})-\phi^{i+1}(e)v & vu+\phi^{i+1}(e)\phi^{i+2}(e) \end{pmatrix} \\
&= \begin{pmatrix} \t(\phi^i\phi^{i+1})+uv & u\s\t(\phi^i)(e-d)-u\nu_f(\phi^{i})(e-d) \\ v\t(\phi^{i+1})-v\t(\phi^{i+1}) & (\phi^{i+1}\phi^{i+2})(e)+uv \end{pmatrix} = \begin{pmatrix} (f+uv)\cdot & 0 \\ 0 & (f+uv)\cdot \end{pmatrix} 
\end{align*}
for every $i\in \ZZ$, so $\cH(\phi)\in \NMF_{S[u; \s][v; \t]}^{\ZZ}(f+uv)$.  For $\mu\in \Hom_{\NMF_S^{\ZZ}(f)}(\phi, \psi)$, since $\mu^{i+2}(e)=\nu_f(\mu^i)(e-d)$
for every $i\in \ZZ$ by 
Theorem \ref{thm.nmf} (2), 
\begin{align*}
& \begin{pmatrix} 
\t(\mu^{i}) & 0 \\ 0 & \mu^{i+1}(e) \end{pmatrix}\begin{pmatrix} \t(\phi^i) & u\cdot  \\ v\cdot  & -\phi^{i+1}(e) \end{pmatrix} \\
= & \begin{pmatrix} \t(\mu^i)\t(\phi^i) & \t(\mu^i)u  \\ \mu^{i+1}(e)v  & -\mu^{i+1}(e)\phi^{i+1}(e) \end{pmatrix} = \begin{pmatrix} \t(\mu^i\phi^i) & u\s\t(\mu^i)(e-d) \\ v\t(\mu^{i+1})  & -(\mu^{i+1}\phi^{i+1})(e) \end{pmatrix} \\
= & \begin{pmatrix} \t(\psi^i\mu^{i+1}) & u\nu_f(\mu^{i})(e-d)  \\ v\t(\mu^{i+1})  & -(\psi^{i+1}\mu^{i+2})(e) \end{pmatrix}=\begin{pmatrix} \t(\psi^i)\t(\mu^{i+1}) & u\mu^{i+2}(e)  \\ v\t(\mu^{i+1})  & -\psi^{i+1}(e)\mu^{i+2}(e) \end{pmatrix} \\
= & \begin{pmatrix} \t(\psi^i) & u\cdot  \\ v\cdot  & -\psi^{i+1}(e) \end{pmatrix}\begin{pmatrix} \t(\mu^{i+1}) & 0 \\ 0 & \mu^{i+2}(e) \end{pmatrix}
\end{align*}
for every $i\in \ZZ$, so we have $\cH(\mu)\in \Hom_{\NMF_{S[u; \s][v; \t]}^{\ZZ}(f+uv)}(\cH(\phi), \cH(\psi))$, hence $\cH:\NMF_S^{\ZZ}(f)\to \NMF_{S[u; \s][v; \t]}^{\ZZ}(f+uv)$ is a functor.  

Since 
\begin{align*}
& \begin{pmatrix}1 & u \\ 0 & -1 \end{pmatrix}\begin{pmatrix} f+uv & 0 \\ 0 & 1 \end{pmatrix}=
\begin{pmatrix} f+uv & u \\ 0 & -1 \end{pmatrix}=
\begin{pmatrix} f & u \\ v & -1 \end{pmatrix}\begin{pmatrix} 1 & 0 \\ v &  1 \end{pmatrix},\\
& \begin{pmatrix} 1 & 0 \\ v & 1 \end{pmatrix}\begin{pmatrix} 1 & 0 \\ 0 & f+uv \end{pmatrix}=\begin{pmatrix} 1 & 0 \\ v & vu+f \end{pmatrix}=
\begin{pmatrix} 1 & u \\ v & -f \end{pmatrix}\begin{pmatrix} 1 & u \\ 0 &  -1 \end{pmatrix}, 
\end{align*}
the diagram
\[\xymatrix@R=4pc@C=5pc{
F^{2i+2}[u; \s][v; \t]\oplus F^{2i+3}[u; \s][v; \t](e)
\ar[r]^{\begin{pmatrix} 1 & 0 \\ v &  1 \end{pmatrix}}_{\cong}
\ar[d]_{\phi_{(F[u; \s][v; \t])}^{2i+1}\oplus {{_{(F[u; \s][v; \t](e))}}\phi}^{2i+1}}
&F^{2i+2}[u; \s][v; \t]\oplus F^{2i+3}[u; \s][v; \t](e)
\ar[d]^{\cH(\phi_F)^{2i+1}} \\
F^{2i+1}[u; \s][v; \t]\oplus F^{2i+2}[u; \s][v; \t](e)
\ar[r]^{\begin{pmatrix} 1 & u \\ 0 & -1 \end{pmatrix}}_{\cong}
\ar[d]_{\phi_{(F[u; \s][v; \t])}^{2i}\oplus {_{(F[u; \s][v; \t](e))}}\phi^{2i}}
&F^{2i+1}[u; \s][v; \t]\oplus F^{2i+2}[u; \s][v; \t](e)
\ar[d]^{\cH(\phi_F)^{2i}} \\
F^{2i}[u; \s][v; \t]\oplus F^{2i+1}[u; \s][v; \t](e)
\ar[r]^{\begin{pmatrix} 1 & 0 \\ v & 1 \end{pmatrix}}_{\cong}
&F^{2i}[u; \s][v; \t]\oplus F^{2i+1}[u; \s][v; \t](e) 
}\]
commutes for every $i\in \ZZ$ where $F^{2i+1}=F^{2i}=F(-id)$,
so $\cH(\phi_{F})\cong \phi_{(F[u; \s][v; \t])}\oplus {_{(F[u; \s][v; \t](e))}}\phi$.
Similarly, we can show that $\cH({_{F}\phi}) \cong {_{(F[u; \s][v; \t])}}\phi \oplus \phi_{(F[u; \s][v; \t](-d+e))}$,
so $\cH$ induces a functor 
$$\underline \cH :\uNMF_S^{\ZZ}(f)\to \uNMF_{S[u; \s][v; \t]}^{\ZZ}(f+uv).$$    
\end{proof}

Let $S$ be a graded ring.  A {\em homogeneous matrix} over $S$ is a matrix whose entries are homogeneous elements in $S$.
For homogeneous matrices $\Phi, \Psi$ over $S$, we write $\Phi \sim \Psi$ if there exist invertible homogeneous matrices $P, Q$ over $S$
such that $\Psi=P\Phi Q$.  If $\phi, \psi$ are graded right $S$-module homomorphisms between graded free $S$-modules given by the left multiplications of $\Phi, \Psi$ respectively, then we write $\phi\sim \psi$ when $\Phi\sim \Psi$.  
Note that $\phi \sim \psi$ if and only if $\Coker \phi \cong \Coker \psi $ as graded $S$-modules.  

Let $f\in S_d$ and $A=S/(f)$.
If $\phi\sim \psi$, then it is easy to see that $\overline \phi\sim \overline \psi$.
For $\phi, \psi\in \NMF_S^{\ZZ}(f)$, if $\phi\cong \psi$, then $\phi^i\sim \psi^i$ for every $i\in \ZZ$.   

Let $f\in S_d$ be a regular normal element, and $\s\in \GrAut S$. 
Note that if $P, Q$ are homogeneous matrices over $S$, then $\s(PQ)=\s(P)\s(Q)$ and $Pf=f\nu_f(P)$.  

\begin{lemma} \label{lem.a1n} 
Let $S$ be a graded ring, $f\in S_d$ a regular normal element, $\s, \t\in \GrAut _0(S; f)$ such that 
$\s\t=\t\s=\nu_f$.  
Let $\mu=\left\{ \begin{pmatrix} \a^i & \b^i \\ \c^i & \d^i \end{pmatrix} \right\}_{i \in \ZZ} \in \Hom_{\NMF_{S[u; \s][v; \t]}^{\ZZ}(f+uv)}(\cH(\phi), \cH(\psi))$ where $\a^i, \b^i, \c^i, \d^i$ are given by left multiplications of the homogeneous matrices $A^i, B^i, C^i, D^i$. 
If $A^1|_{u=v=0}=0$ in $S$,
then 
$$
\begin{pmatrix} 
\t(\psi^1) & u\cdot & \a^1 & \b^1 \\ 
v\cdot & -\psi^2(e) & \c^1& \d^1 \\
0 & 0 & \t(\phi^0) & u\cdot \\
0 & 0 & v\cdot & -\phi^1(e) 
\end{pmatrix}
\sim
\begin{pmatrix} 
\t(\psi^1) & u\cdot & 0 & 0 \\ 
v\cdot & -\psi^2(e) & 0& 0 \\
0 & 0 & \t(\phi^0) & u\cdot \\
0 & 0 & v\cdot & -\phi^1(e) 
\end{pmatrix}.$$
\end{lemma} 

\begin{proof} By Remark \ref{rem.uvf}, we write $\nu_u=\s, \nu_v=\t, \nu_{f+uv}=\nu\in \GrAut (S[u; \s][v; \t])$ by abuse of notation.
Suppose that $\phi^i, \psi^i$ are given by the left multiplications of the homogeneous matrices $\Phi^i, \Psi^i$.
We should calculate
\[
\begin{pmatrix} 
\t(\Psi^1) & u & A^1 & B^1 \\ 
v & -\Psi^2 & C^1& D^1 \\
0 & 0 & \t(\Phi^0) & u \\
0 & 0 & v & -\Phi^1 
\end{pmatrix}.
\]
Since $A^1|_{u=v=0}=0$ in $S$, it follows that $A^1=uR+R'v$ for some homogeneous matrices $R, R'$ over $S[u; \s][v; \t]$, so we may assume that $A^1=0$ by a (noncommutative) elementary transformation of matrices.  Since $A^{2i+1}=\nu^i(A^1)=0$,  we have 
\begin{align*}
& \begin{pmatrix} A^{2i} & B^{2i}  \\ C^{2i}  & D^{2i} \end{pmatrix}\begin{pmatrix} \t(\Phi^{2i}) & u  \\ v  & -\Phi^{2i+1} \end{pmatrix}=\begin{pmatrix} \t(\Psi^{2i}) & u  \\ v  & -\Psi^{2i+1} \end{pmatrix}\begin{pmatrix} 0 & B^{2i+1}  \\ C^{2i+1}  & D^{2i+1} \end{pmatrix} \\
& \begin{pmatrix} 0 & B^{2i+1}  \\ C^{2i+1}  & D^{2i+1} \end{pmatrix}\begin{pmatrix} \t(\Phi^{2i+1}) & u  \\ v  & -\Phi^{2i+2} \end{pmatrix}=\begin{pmatrix} \t(\Psi^{2i+1}) & u  \\ v  & -\Psi^{2i+2} \end{pmatrix}\begin{pmatrix} A^{2i+2} & B^{2i+2}  \\ C^{2i+2}  & D^{2i+2} \end{pmatrix}.
\end{align*}
In particular, 
\begin{align}
& B^1v=\t(\Psi^1)A^2+uC^2 \label{eq1}\\ 
& uC^3=A^2\t(\Phi^2)+B^2v \label{eq2}\\
& \t(\Psi^0)B^1+uD^1=A^0u-B^0\Phi^1. \label{eq3}
\end{align}

Since
$fA^2=\t(f)A^2=\t(\Psi^0)\t(\Psi^1)A^2=\t(\Psi^0)(B^1v-uC^2)$
by (\ref{eq1}), it follows that $fA^2|_{u=v=0}=0$ in $S$, so we have $A^2|_{u=v=0}=0$ in $S$.
Thus 
\begin{align}
A^2=uP+vQ \textnormal {\; for some homogeneous matrices \;} P, Q. \label{eqpq}
\end{align}

By (\ref{eq1}) and (\ref{eqpq}),
$B^1v=\t(\Psi^1)(uP+vQ)+uC^2$,  
so 
$v(\t(B^1)-\t^2(\Psi^1)Q)=u(\Psi^3P+C^2)$,  
and hence we see 
\begin{align}
\t(B^1)-\t^2(\Psi^1)Q=uP', \;\; \Psi^3P+C^2=vP' \quad \textnormal {\; for some homogeneous matrix \;} P'. \label{eqp'}
\end{align}

By (\ref{eq2}) and (\ref{eqpq}), $uC^3=A^2\t(\Phi^2)+B^2v=(uP+vQ)\t(\Phi^2)+B^2v$, 
so 
$u(C^3-P\t(\Phi^2))=v(Q\t(\Phi^2)+\t(B^2))$,  
and hence we see 
\begin{align}
C^3-P\t(\Phi^2)=vQ', \;\; Q\t(\Phi^2)+\t(B^2)=uQ' \quad \textnormal {\; for some homogeneous matrix \;} Q'. \label{eqq'}
\end{align}

Since
\begin{align*}
& -\t(\Psi^1)\t^{-1}(Q)-u\t^{-1}(P')+B^1=\t^{-1}(-\t^2(\Psi^1)Q-uP'+\t(B^1))=0  \quad \textnormal{ by (\ref{eqp'}) }, \\
& C^1-\nu^{-1}(P)\t(\Phi^0)-\nu^{-1}\t^{-1}(Q')v=\nu^{-1}(C^3-P\t(\Phi^2)-vQ')=0  \quad \textnormal{ by (\ref{eqq'}) },
\end{align*}
and
\begin{align*}
&u(-v\t^{-1}(Q)+\Psi^2\t^{-1}(P')+D^1-\nu^{-1}(P)u+\t^{-1}\nu^{-1}(Q')\Phi^1) \\
& =-uv\t^{-1}(Q)-u^2\t^{-1}(P)+uD^1+\s^{-1}(\Psi^2)u\t^{-1}(P')+u\t^{-1}\nu^{-1}(Q')\Phi^1 \\
& =-u\t^{-1}(A^2)+uD^1+\t(\Psi^0)\t^{-1}(uP')+\t^{-1}\nu^{-1}(uQ')\Phi^1 \quad \textnormal{ by (\ref{eqpq})} \\
& =-u\s(A^0)+uD^1+\t(\Psi^0)(B^1-\t(\Psi^1)\t^{-1}(Q))+(\t^{-1}\nu^{-1}(Q)\Phi^0+B^0)\Phi^1 \quad \textnormal{ by (\ref{eqp'}), (\ref{eqq'})}\\
& =-A^0u+uD^1+\t(\Psi^0)B^1-f\t^{-1}(Q)+\t^{-1}\nu^{-1}(Q)f+B^0\Phi^1 \\
& = -f\t^{-1}(Q)+\t^{-1}\nu^{-1}(Q)f  \quad \textnormal{ by (\ref{eq3})}\\
&=0
\end{align*}
we obtain
{\small
\begin{align*}
& \begin{pmatrix} 
1 & 0 & 0 & 0 \\ 
0 & 1 & -\nu^{-1}(P) & -\t^{-1}\nu^{-1}(Q') \\
0 & 0 & 1 & 0 \\
0 & 0 & 0 & 1 
\end{pmatrix}\,\
\begin{pmatrix} 
\t(\Psi^1) & u & 0 & B^1 \\ 
v & -\Psi^2 & C^1& D^1 \\
0 & 0 & \t(\Phi^0) & u \\
0 & 0 & v & -\Phi^1 
\end{pmatrix}\,\
\begin{pmatrix} 
1 & 0 & 0 & -\t^{-1}(Q) \\ 
0 & 1 & 0 & -\t^{-1}(P') \\
0 & 0 & 1 & 0 \\
0 & 0 & 0 & 1 
\end{pmatrix} \\
= & \begin{pmatrix} 
\t(\Psi^1) & u & 0 & B^1 \\ 
v & -\Psi^2 & C^1-\nu^{-1}(P)\t(\Phi^0)-\t^{-1}\nu^{-1}(Q')v & D^1-\nu^{-1}(P)u+\t^{-1}\nu^{-1}(Q')\Phi^1 \\
0 & 0 & \t(\Phi^0) & u \\
0 & 0 & v & -\Phi^1 
\end{pmatrix} \,\
\begin{pmatrix} 
1 & 0 & 0 & -\t^{-1}(Q) \\ 
0 & 1 & 0 & -\t^{-1}(P') \\
0 & 0 & 1 & 0 \\
0 & 0 & 0 & 1 
\end{pmatrix} \\
= & \begin{pmatrix} 
\t(\Psi^1) & u & 0 & -\t(\Psi^1)\t^{-1}(Q)-u\t^{-1}(P')+B^1 \\ 
v & -\Psi^2 & C^1-\nu^{-1}(P)\t(\Phi^0)-\t^{-1}\nu^{-1}(Q')v & -v\t^{-1}(Q)+\Psi^2\t^{-1}(P')+D^1-\nu^{-1}(P)u+\t^{-1}\nu^{-1}(Q')\Phi^1 \\
0 & 0 & \t(\Phi^0) & u \\
0 & 0 & v & -\Phi^1 
\end{pmatrix}\\
=& \begin{pmatrix} 
\t(\Psi^1) & u & 0 & 0 \\ 
v & -\Psi^2 & 0 & 0 \\
0 & 0 & \t(\Phi^0) & u \\
0 & 0 & v & -\Phi^1 
\end{pmatrix}.
\end{align*}
}Hence the assertion follows.
\end{proof} 

\begin{lemma} \label{lem.35} 
Let $S$ be a right noetherian graded ring, $f\in S_d$ a regular normal element, and $A=S/(f)$.  
\begin{enumerate} 
\item{} $\Coker:\NMF_S^{\ZZ}(f)\to \TR^{\ZZ}(A)$ induces a fully faithful functor $\underline {\Coker}:\uNMF_S^{\ZZ}(f)\to \uTR^{\ZZ}(A)$.
\item{} For $\mu\in \Hom_{\NMF_S^{\ZZ}(f)}(\phi, \psi)$,
we have $e_S(\mu):=\left\{\begin{pmatrix} \psi^{i+1} & (-1)^{i}\mu^{i+1} \\ 0 & \phi^{i} \end{pmatrix}\right \}_{i\in \ZZ} \in \NMF_S^{\ZZ}(f)$. 
\item{} 
The map $\rho_S :\Hom_{\NMF_S^{\ZZ}(f)}(\phi, \psi)\to \Ext^1_{\GrMod A}(\Coker \phi, \Omega\Coker \psi)$ defined by 
$$\rho_S (\mu):=\{0\to \Omega \Coker \psi\to \Coker e_S(\mu)\to \Coker \phi\to 0\}$$ 
induces a bijection $\underline{\rho_S} :\Hom_{\uNMF_S^{\ZZ}(f)}(\phi, \psi)\to \Ext^1_{\GrMod A}(\Coker \phi, \Omega\Coker \psi)$. 
\item{} For $\mu\in \Hom_{\NMF_S^{\ZZ}(f)}(\phi, \psi)$, $\rho_S(\mu)=0$ if and only if 
$\overline{e_S(\mu)^0}= \overline{\begin{pmatrix} \psi^1& \mu^1 \\ 0 & \phi^0 \end{pmatrix}}
\sim \overline{\begin{pmatrix} \psi^1 & 0 \\ 0 & \phi^0 \end{pmatrix}}$.
\end{enumerate}
\end{lemma} 

\begin{proof} 
(1) This follows from \cite [Theorem 6.6]{MU}.

(2) Since $\mu^i\phi^i=\psi^i\mu^{i+1}$ for every $i\in \ZZ$, 
$$\begin{pmatrix} \psi^{i+1} & (-1)^{i}\mu^{i+1} \\ 0 & \phi^{i} \end{pmatrix}\begin{pmatrix} \psi^{i+2} & (-1)^{i+1}\mu^{i+2} \\ 0 & \phi^{i+1} \end{pmatrix}=\begin{pmatrix} \psi^{i+1}\psi^{i+2} & (-1)^{i}(\mu^{i+1}\phi^{i+1}-\psi^{i+1}\mu^{i+2}) \\ 0 & \phi^{i}\phi^{i+1}\end{pmatrix}=\begin{pmatrix} f & 0 \\ 0 & f \end{pmatrix}$$
for every $i\in \ZZ$, so $e_S(\mu):=\left\{\begin{pmatrix} \psi^{i+1} & (-1)^{i}\mu^{i+1} \\ 0 & \phi^{i}\end{pmatrix}\right \}_{i\in \ZZ}\in \NMF_S^{\ZZ}(f)$. 

(3) 
Since $C(\phi)^{\geq 0}:=\{\overline{\phi^i}: \overline{F^{i+1}} \to \overline{F^i}\}_{i\geq 0}$ and $C(\psi)^{\geq 0}:=\{\overline{\psi^i}: \overline{G^{i+1}} \to \overline{G^i}\}_{i\geq 0}$ are free resolutions of $\Coker \phi$ and $\Coker \psi$ in $\GrMod A$, 
$$\delta(\Coker (\mu))=\{0\to \Omega \Coker \psi\to \Coker \overline{\begin{pmatrix} \psi^1 & \mu^1 \\ 0 & \phi^0 \end{pmatrix}}=\Coker e_S(\mu)
\to \Coker \phi\to 0\},$$ 
so $\underline{\rho_S}:\Hom_{\uNMF_S^{\ZZ}(f)}(\phi, \psi)\to \Ext^1_{\GrMod A}(\Coker \phi, \Omega\Coker \psi)$ is 
given by the composition of maps
$$\begin{CD} \Hom_{\uNMF_S^{\ZZ}(f)}(\phi, \psi) @>\underline {\Coker}>> \uHom_{\GrMod A}(\Coker \phi, \Coker \psi)@>\underline{\delta}>> \Ext^1_{\GrMod A}(\Coker \phi, \Omega \Coker \psi).\end{CD}$$
By (1), $\underline{\Coker}$ is bijective.  Since $\Coker \phi, \Coker \psi\in \TR^{\ZZ}(A)$, $\underline{\delta}$ is bijective by Lemma \ref{lem.ch},
so we have the result.  

(4) 
For $\mu\in \Hom_{\NMF_S^{\ZZ}(f)}(\phi, \psi)$, $\rho_S(\mu)=0$ if and only if 
\[
\Coker \overline{\begin{pmatrix} \psi^1& \mu^1 \\ 0 & \phi^0 \end{pmatrix}} = \Coker e_S(\mu) \cong \Omega \Coker \psi \oplus \Coker \phi\cong \Coker \overline{\begin{pmatrix} \psi^1& 0 \\ 0 & \phi^0 \end{pmatrix}}
\]
if and only if $\overline{e_S(\mu)^0} =\overline{\begin{pmatrix} \psi^1& \mu^1 \\ 0 & \phi^0 \end{pmatrix}}
\sim \overline{\begin{pmatrix} \psi^1& 0 \\ 0 & \phi^0 \end{pmatrix}}$. 
\end{proof} 

We are now ready to prove a noncommutative graded version of \cite[Proposition 3.1]{Sol}.

\begin{theorem} \label{thm.ff} 
If $S$ is a right noetherian graded ring, $f\in S_d$ is a regular normal element, and $\s, \t\in \GrAut _0(S; f)$ such that 
$\s\t=\t\s=\nu_f$, 
then  
$\underline{\cH}:\uNMF_S^{\ZZ}(f)\to \uNMF_{S[u; \s][v; \t]}^{\ZZ}(f+uv)$ is a fully faithful functor.
\end{theorem} 
 
\begin{proof} 
We first note that if $A$ is a matrix whose entries are in $S[u; \s][v; \t]$, then $A|_{u=v=0}$ is a matrix whose entries are in $S$.
In this proof, if $\a$ is a map given by left multiplication of $A$, then we denote by $\a|_{u=v=0}$ a map given by left multiplication of $A|_{u=v=0}$ by abuse of notation.

For $\mu\in \Hom_{\NMF_S^{\ZZ}(f)}(\phi, \psi)$,
if $\underline{\cH(\mu)}=\underline{\cH}(\underline{\mu})=0$ in $\uNMF_{S[u; \s][v; \t]}^{\ZZ}(f+uv)$,
then we have $\rho_{S[u; \s][v; \t]}(\cH(\mu))=0$ by Lemma \ref{lem.35} (3), so 
\begin{align*}
\overline{e_{S[u; \s][v; \t]}(\cH(\mu))^0}
=\overline{\begin{pmatrix} \cH(\psi)^1 & \cH(\mu)^1 \\ 0 & \cH(\phi)^0 \end{pmatrix}}
\sim \overline{\begin{pmatrix} \cH(\psi)^1 & 0 \\ 0 & \cH(\phi)^0 \end{pmatrix}}
\end{align*}
over $S[u; \s][v; \t]/(f+uv)$ by Lemma \ref{lem.35} (4).
Thus we have
\begin{align*}
\begin{pmatrix}
\overline{\t(\psi^1)} & \overline{u} & \overline{\t(\mu^1)} & 0 \\ \overline{v} & \overline{-\psi^2}(e) & 0 & \overline{\mu^2}(e) \\ 0 & 0 & \overline{\t(\phi^0)} & \overline{u} \\ 0 & 0 & \overline{v} & \overline{-\phi^1}(e) \end{pmatrix}
\sim \begin{pmatrix} \overline{\t(\psi^1)} & \overline{u} & 0 & 0 \\ \overline{v} & \overline{-\psi^2}(e) & 0 & 0 \\ 0 & 0 & \overline{\t(\phi^0)} & \overline{u} \\ 0 & 0 & \overline{v} & \overline{-\phi^1}(e) \end{pmatrix}
\end{align*}
over $S[u; \s][v; \t]/(f+uv)$ where $\deg v=e$. 
By switching the second row and the third row, and the second column and the third column, and substituting $\overline{u}=\overline{v}=0$, 
\[
\begin{pmatrix}
\overline{\t(\psi^1)} & \overline{\t(\mu^1)} & 0 & 0 \\ 0 & \overline{\t(\phi^0)} & 0 & 0 \\ 0 & 0 & \overline{-\psi^2}(e) & \overline{\mu^2}(e)  \\ 0 & 0 & 0 & -\overline{\phi^1}(e)
\end{pmatrix}
\sim
\begin{pmatrix}
\overline{\t(\psi^1)} & 0 & 0 & 0 \\ 0 & \overline{\t(\phi^0)} & 0 & 0 \\ 0 & 0 & \overline{-\psi^2}(e) & 0 \\ 0 & 0 & 0 & \overline{-\phi^1}(e)
\end{pmatrix}
\] 
over $S/(f)$.  It follows that 
$\begin{pmatrix} \overline{\t(\psi^1)} & \overline{\t(\mu^1)} \\ 0 & \overline{\t(\phi^0)} \end{pmatrix} \sim \begin{pmatrix} \overline{\t(\psi^1)} & 0 \\ 0 & \overline{\t(\phi^0)} \end{pmatrix}$,
so 
$$\overline{e_S(\mu)^0}
=\overline{\begin{pmatrix} \psi^1 & \mu^1 \\ 0 & \phi^0 \end{pmatrix}}
=\begin{pmatrix} \overline{\psi^1} & \overline{\mu^1} \\ 0 & \overline{\phi^0} \end{pmatrix}
\sim
\begin{pmatrix} \overline{\psi^1} & 0 \\ 0 & \overline{\phi^0} \end{pmatrix}
=\overline{\begin{pmatrix} \psi^1 & 0 \\ 0 & \phi^0 \end{pmatrix}}$$
over $S/(f)$, hence $\rho_S(\mu)=0$ by Lemma \ref{lem.35} (4) again.  By Lemma \ref{lem.35} (3) again, 
$\underline{\mu}=0$, so $\underline{\cH}$ is faithful.

For $\mu=\begin{pmatrix} \a & \b \\ \c & \d \end{pmatrix} = \left\{ \begin{pmatrix} \a^i & \b^i \\ \c^i & \d^i \end{pmatrix} \right\}_{i \in \ZZ} \in \Hom_{\NMF_{S[u; \s][v; \t]}^{\ZZ}(f+uv)}(\cH(\phi), \cH(\psi))$, we have
$$\begin{pmatrix} \a^i & \b^i  \\ \c^i  & \d^i \end{pmatrix}\begin{pmatrix} \t(\phi^i) & u\cdot  \\ v\cdot  & -\phi^{i+1}(e) \end{pmatrix}=\begin{pmatrix} \t(\psi^i) & u\cdot  \\ v\cdot  & -\psi^{i+1}(e) \end{pmatrix}\begin{pmatrix} \a^{i+1} & \b^{i+1}  \\ \c^{i+1}  & \d^{i+1} \end{pmatrix},$$
so $\t^{-1}(\a^i)|_{u=v=0}\phi^i=\psi^i\t^{-1}(\a^{i+1})|_{u=v=0}$,
hence we have $\t^{-1}(\a)|_{u=v=0} \in \Hom_{\NMF_S^{\ZZ}(f)}(\phi, \psi)$.
Since 
$$\mu-\cH(\t^{-1}(\a)|_{u=v=0})=
\left\{\begin{pmatrix} \a^i-\a^i|_{u=v=0} & \b^i \\ \c^i & \d^i-\t^{-1}(\a^{i+1})|_{u=v=0}(e) \end{pmatrix}\right\}_{i\in \ZZ},$$
where $(\a^1-\a^1|_{u=v=0})|_{u=v=0}=0$ in $S$, it follows that 
\begin{align*}
& e_{S[u; \s][v; \t]}(\mu-\cH(\t^{-1}(\a)|_{u=v=0}))^0= \\
& \begin{pmatrix} 
\t(\psi^1) & u\cdot & \a^1-\alpha^1|_{u=v=0} & \b^1 \\ 
v\cdot & -\psi^2(e) & \c^1& \d^1-\t^{-1}(\a^2)|_{u=v=0}(e) \\
0 & 0 & \t(\phi^0) & u\cdot \\
0 & 0 & v\cdot & -\phi^1(e) 
\end{pmatrix}\sim
\begin{pmatrix} 
\t(\psi^1) & u\cdot & 0 & 0 \\ 
v\cdot  & -\psi^2(e) & 0& 0 \\
0 & 0 & \t(\phi^0) & u\cdot \\
0 & 0 & v\cdot & -\phi^1(e) 
\end{pmatrix}
\end{align*}
by Lemma \ref{lem.a1n}, so
$\rho_{S[u; \s][v; \t]}(\mu-\cH(\t^{-1}(\a)|_{u=v=0}))=0$ by Lemma \ref{lem.35} (4). 
By Lemma \ref{lem.35} (3), $\underline{\mu}=\underline{\cH}(\underline{\t^{-1}(\a)|_{u=v=0}})$, so $\underline{\cH}$ is full.   
\end{proof}

\begin{theorem}[Noncommutative Kn\"orrer's periodicity theorem] \label{thm.nkp} 
Suppose that $k$ is an algebraically closed field of characteristic not 2.
Let $S$ be a noetherian AS-regular algebra and $f\in S$ a regular normal homogeneous element of even degree.  If there exists $\s\in \GrAut _0(S; f)$ such that 
$\s^2=\nu_f$ (eg. if $f$ is central and $\s=\id$), then 
$\underline{\cH}:\uNMF_S^{\ZZ}(f)\to \uNMF_{S[u; \s][v; \s]}^{\ZZ}(f+uv)$ is an equivalence functor where $\deg u=\deg v=\frac{1}{2}\deg f$.   Moreover, we have $\uCM^{\ZZ}(S/(f))\cong \uCM^{\ZZ}(S[u; \s][v; \s]/(f+u^2+v^2))$.  
\end{theorem} 

\begin{proof}
By Theorem \ref{thm.ff}, $\underline{\cH}:\uNMF_S^{\ZZ}(f)\to \uNMF_{S[u; \s][v; \s]}^{\ZZ}(f+uv)$ is fully faithful.  By \cite[Theorem 5.11]{CKMW} and \cite[Proposition 4.7]{MU}, $\underline{\cH}:\uNMF_S^{\ZZ}(f)\to \uNMF_{S[u; \s][v; \s]}^{\ZZ}(f+uv)$ is dense, so $\underline{\cH}:\uNMF_S^{\ZZ}(f)\to \uNMF_{S[u; \s][v; \s]}^{\ZZ}(f+uv)$ is an equivalence functor.  

Since $k$ is an algebraically closed field of characteristic not 2, 
there exists $\theta\in \GrAut _0(S[u; \s][v; \s])$ defined by $\theta|_S=\id_S, 
\theta(u)=u+\sqrt{-1}v, \theta(v)=u-\sqrt{-1}v$ such that $\theta(f+uv)=f+u^2+v^2$, so 
\begin{align*}
\uCM^{\ZZ}(S/(f)) & \cong \uNMF_S^{\ZZ}(f) \cong \uNMF_{S[u; \s][v; \s]}^{\ZZ}(f+uv) \\
& \cong \uNMF_{S[u; \s][v; \s]}^{\ZZ}(f+u^2+v^2)\cong 
\uCM^{\ZZ}(S[u; \s][v; \s]/(f+u^2+v^2))
\end{align*}
by Theorem \ref{thm.m4} and Lemma \ref{lem.eq1}.  
\end{proof} 

\begin{remark} Note that, just forgetting the grading, all the results in this subsection except for Theorem \ref{thm.nkp}, which highly depends on the results in \cite {CKMW}, hold in the ungraded setting.
It is curious if there is an ungraded version of Theorem \ref{thm.nkp}.
It seems that the following two conditions are essential for the arguments in \cite {CKMW} to work, namely,
(1) $A$ is ``local'' so that every finitely generated projective module is free, and
(2) $A$ is ``complete'' so that $\uCM(A)$ is Krull-Schmidt.
The authors do not know a nice class of ungraded noncommutative algebras satisfying both conditions.  
\end{remark} 


\section{Noncommutative Quadric Hypersurfaces} 

In this section, we define a noncommutative quadric hypersurface $A$ and show that there is another way to reduce the number of variables in computing $\uCM^{\ZZ}(A)$. 

\subsection{Strongly Graded Algebras} 

We first list some basic properties of a strongly graded algebra which are needed in this paper.  
Most of the results in this subsection are well-known.
We will give some of the proofs for the convenience of the reader.

\begin{definition}
A graded algebra $A$ is called \emph{strongly graded} if $A_iA_j=A_{i+j}$ for all $i, j\in \ZZ$. 
\end{definition} 
 
\begin{lemma}[{\cite[Theorem \Rnum{1}. 3.4]{NV}}] \label{lem.sg} 
If $A$ is a strongly graded algebra, then $\GrMod A\to \Mod A_0; \; M\mapsto M_0$ and $\Mod A_0\to \GrMod A; \; N\mapsto N\otimes _{A_0}A$ are equivalence functors inverses to each other. 
\end{lemma}

\begin{lemma}[{\cite[Lemma \Rnum{2}. 3.6]{NV}}] \label{lem.noe1} 
Let $A$ be a strongly graded ring.  Then $A$ is right noetherian if and only if $A_0$ is right noetherian.   
\end{lemma} 


A typical example of a strongly graded algebra is obtained by the localization by a regular normal element.  If $A$ is a ring and $w\in A$ is a regular normal element, then 
$A[w^{-1}]:=\{aw^{-i}\mid a\in A, i\in \NN\}$ is a ring under the addition $aw^{-i}+bw^{-j}=(aw^j+bw^i)w^{-i-j}$ and the multiplication $(aw^{-i})(bw^{-j})=a\nu_w^i(b)w^{-i-j}$.
If $A$ is a graded ring and $w\in A_d$ is a homogeneous regular normal element, then $A[w^{-1}]$ is a graded ring by $\deg(aw^{-i})=\deg a-id$. 
In this case, $A[w^{-1}]_0=\{aw^{-i}\mid a\in A_{id}, i\in \NN\}$.
The proofs of the next two lemmas are straight-forward and omitted.

\begin{lemma} \label{lem.wsg} 
If $A$ is a graded algebra generated in degree 1 over $k$, and $w\in A_d$ is a homogeneous regular normal element of positive degree, then $A[w^{-1}]$ is a strongly graded algebra. 
\end{lemma}

\begin{lemma} \label{lem.noe2} 
Let $A$ be a (graded) ring and $w\in A$ a (homogeneous) regular normal element. If $A$ is a right noetherian ring, then $A[w^{-1}]$ is also a right noetherian ring. 
\end{lemma} 

\begin{proposition} \label{prop.sg} Let $A$ be a right noetherian 
graded algebra generated in degree 1 over $k$ and $w\in A_d$ a homogeneous regular normal element of positive degree.  If $\dim_kA/(w)<\infty$, then $\tails A\to \mod A[w^{-1}]_0; \; \pi M\mapsto  M[w^{-1}]_0$ is an equivalence functor. 
\end{proposition} 

\begin{proof} Let $\phi:A\to A[w^{-1}]$ be the canonical injection. First, we show that  $\phi^*:\GrMod A\to \GrMod A[w^{-1}]$ defined by $\phi^*(M)=M\otimes _AA[w^{-1}]$ induces an equivalence functor $\Tails A\to \GrMod A[w^{-1}]$.  In fact, since $\phi^*\phi_*\cong \Id_{\GrMod A[w^{-1}]}$ where $\phi_*:\GrMod A[w^{-1}]\to \GrMod A$ is the restriction functor, $\phi^*$ is dense, so it is enough to show that $\Ker \phi^*=\Tors A=\Ker \pi$.   Clearly, $\Ker \pi \cap \grmod A=\tors A \subset \Ker \phi^*$.  Suppose that $M\in \Ker \phi^*\cap \grmod A$.  Since $\phi^*(M)=M[w^{-1}]=0$, for every $m\in M$, there exists $i\in \NN$ such that $mw^i=0$.  Since $M\in \grmod A$, there exists $r\in\NN$ such that $Mw^r=0$.  Consider the filtration 
$$0=Mw^r\subset Mw^{r-1}\subset \cdots \subset Mw\subset M.$$
Since $(Mw^{i-1}/Mw^{i})w=0$, it follows that $Mw^{i-1}/Mw^{i}\in \grmod A/(w)$ for $i=1, \dots, r$. Since $\dim _kA/(w)<\infty$, it follows that $\dim _kM<\infty$, so $M\in \Ker \pi$.  We have so far proved that $\Ker \phi^*\cap \grmod A= \Ker \pi\cap \grmod A$.  Since every module $M\in \GrMod A$ is a direct limit of finitely generated modules, and both functors $\phi^*$ and $\pi$ commute with direct limits, we have $\Ker \phi^*= \Ker \pi$, so $\Tails A\to \GrMod A[w^{-1}]$ is an equivalence functor.

Since $A$ is generated in degree 1 over $k$, $A[w^{-1}]$ is a strongly graded algebra by Lemma \ref{lem.wsg}, so the functor $\GrMod A[w^{-1}]\to \Mod A[w^{-1}]_0; \; N\mapsto N_0$ is an equivalence functor by Lemma \ref{lem.sg}, hence $\Tails A\to \Mod A[w^{-1}]_0; \; \pi M\mapsto  M[w^{-1}]_0$ is an equivalence functor.  Since $A$ is right noetherian, $A[w^{-1}]$ is right noetherian by Lemma \ref{lem.noe2}, so $A[w^{-1}]_0$ is right noetherian by  Lemma \ref{lem.noe1}.  Since an equivalence functor preserves noetherian objects, it restricts to an equivalence functor $\tails A\to \mod A[w^{-1}]_0; \; \pi M\mapsto  M[w^{-1}]_0$. 
\end{proof} 

We list a few more properties of the localization which are needed in this paper.  

\begin{lemma} \label{lem.dimc} 
Let $A$ be a locally finite graded algebra generated in degree 1 over $k$,
$w\in A_d$ a homogeneous regular normal element of positive degree $d$, and $S=A/(w)$.
If $\dim _kS<\infty$, then $\dim_kA[w^{-1}]_0=\dim_kS^{(d)}$ where $S^{(d)}$
is the $d$th Veronese of $S$. 
\end{lemma} 

\begin{proof}
The proof is similar to that of \cite[Lemma 5.1 (3)]{SV}.  
\end{proof}  

\begin{lemma} \label{lem.nqh1} 
Let $A$ be a graded algebra.  For a homogeneous regular normal element $w\in A_d$ and a positive integer $m\in \NN^+$, $A[(w^m)^{-1}]_0=A[w^{-1}]_0$. 
\end{lemma}  

Let $A$ be a ring.  If $u\in A$ is a central element, and $a\in A$, then we write $a/u$ for $au^{-1}=u^{-1}a\in A[u^{-1}]$.  

\begin{lemma}\label{lem.loce}
If $A$ is an $\NN$-graded ring and $u\in A_1$ is a regular central element, then the following hold. 
\begin{enumerate}
\item{} $(A[v])[u^{-1}]_0=A[u^{-1}]_0[v/u]$ where $\deg v=1$.   
\item{} For a central element $w\in A_d$, $(A/(w))[u^{-1}]_0\cong A[u^{-1}]_0/(w/u^d)$. 
\end{enumerate}
\end{lemma} 

\begin{proof}
(1) We have
\begin{align*}
(A[v])[u^{-1}]_0
&=\{f/u^j\mid f\in A[v]_j, j\in \NN\}\\
&=\{(\sum _{i=0}^ja_iv^i)/u^j=\sum_{i=0}^j(a_i/u^{j-i})(v/u)^i \mid a_i\in A_{j-i}, j\in \NN\}\\
&=A[u^{-1}]_0[v/u].
\end{align*}

(2) An exact sequence $0\to (w)\to A\to A/(w)\to 0$ induces an exact sequence 
$0\to (w)[u^{-1}]_0\to A[u^{-1}]_0\to (A/(w))[u^{-1}]_0\to 0.$
Since 
$(w)[u^{-1}]_0=\{(aw)/u^{d+j}=(a/u^j)(w/u^d)\mid a\in A_j, j\in \NN\}=(w/u^d)$
as an ideal of $A[u^{-1}]_0$,  
the result follows.  
\end{proof} 

\subsection{Noncommutative Quadric Hypersurfaces} 

\begin{definition} 
A graded algebra $A$ is called a \emph{homogeneous coordinate ring of a quadric hypersurface 
in a quantum $\PP^{n-1}$}
if $A=S/(f)$ where 
\begin{itemize}
\item $S$ is a quantum polynomial algebra of dimension $n$, and
\item $f\in S_2$ is a regular normal element. 
\end{itemize}
\end{definition}

In the above definition, we do not assume that $f$ is central as in \cite{SV}.
In this section and the next, we will establish some of the results of \cite {SV} in the case that $f$ is normal, rather than central.

By Lemma \ref{lem.bal}, if $A$ is a homogeneous coordinate ring of a quadric hypersurface in a quantum $\PP^{n-1}$, then $A$ is a noetherian AS-Gorenstein algebra of dimension $n-1$.

\begin{lemma}\label{lem.Kos}
If $A=S/(f)$ is a homogeneous coordinate ring of a quadric hypersurface in a quantum $\PP^{n-1}$, then the following hold. 
\begin{enumerate}
\item $S$ and $A$ are Koszul.
\item There exists a unique regular normal element $w \in A_2^!$ up to scalar such that $A^!/(w) = S^!$.
\end{enumerate}
\end{lemma}

\begin{proof}
(1) This follows from \cite[Theorem 5.11]{S} and \cite[Theorem 1.2]{ST}. 

(2) This follows from \cite[Corollary 1.4]{ST}.
\end{proof}

If $A=S/(f)$ is a homogeneous coordinate ring of a quadric hypersurface in a quantum $\PP^{n-1}$, and $w\in A^!_2$ such that $A^!/(w)=S^!$ as above, then we define 
\[ C(A) := A^![w^{-1}]_0. \]

If $S$ is a graded algebra, then the map $S\to S; \; a\mapsto (-1)^{\deg a}a$ is a graded algebra automorphism, which is denoted by $-1\in \GrAut S$ by abuse of notation.  
For example, $k[x, y][u; -1]\cong k\<x, y, u\>/(xy-yx, xu+ux, yu+uy)$.  

Let $S$ be a quantum polynomial algebra.
For a regular central element $f\in S_2$ and $A=S/(f)$, we define $S^{\dagger}=S[u; -1], A^{\dagger}=S^{\dagger}/(f+u^2)$ where $\deg u=1$.  Since $f+u^2\in S^{\dagger}_2$ is a regular central element, we further define $S^{\dagger\dagger}=(S^{\dagger})^{\dagger}=(S[u; -1])[v; -1], A^{\dagger\dagger}=(A^{\dagger})^{\dagger}=S^{\dagger\dagger}/(f+u^2+v^2)$.  
It is easy to see that $(S^{\dagger})^!\cong S^![u]/(u^2)$ and $(S^{\dagger\dagger})^!\cong (S^{\dagger})^![v]/(v^2)\cong S^![u, v]/(u^2, v^2)$ where we write $u, v$ for the duals of $u, v$ by abuse of notation. 

Recall that if $A$ is a Koszul algebra, then $H_A(t)H_{A^!}(-t)=1$.  

\begin{theorem} \label{thm.nqh2} 
If $A=S/(f)$ is a homogeneous coordinate ring of a quadric hypersurface in a quantum $\PP^{n-1}$ where $f\in S_2$ is a regular central element, and $w\in A^!_2$ such that $S^!=A^!/(w)$, then the following hold.
\begin{enumerate}
\item{} $(A^{\dagger})^!\cong A^![u]/(u^2-w)$ and $(A^{\dagger\dagger})^!\cong A^![u, v]/(u^2-w, v^2-w)\cong (A^{\dagger})^![v]/(v^2-u^2)$ after adjusting $w$ by a suitable scalar.
\item{} $w=u^2\in (A^{\dagger})^!_2$ is a regular central element such that $(S^{\dagger})^!=(A^{\dagger})^!/(w)$, and $w=v^2=u^2\in (A^{\dagger\dagger})^!_2$ is a regular central element such that $(S^{\dagger\dagger})^!=(A^{\dagger\dagger})^!/(w)$.
\item{}
$C(A^{\dagger})\cong (A^{\dagger})^![u^{-1}]_0$ and
$C(A^{\dagger\dagger})\cong (A^{\dagger\dagger})^![v^{-1}]_0\cong (A^{\dagger\dagger})^![u^{-1}]_0\cong C(A^{\dagger})^{\times 2}$. 
\end{enumerate}
\end{theorem} 

\begin{proof} (1) Since $A^{\dagger}/(u)=S[u; -1]/(f+u^2, u)\cong S/(f)=A$, the canonical surjection $A^{\dagger}\to A$ induces an injection $A^!\to (A^{\dagger})^!$.  Since $u$ is central in $(A^{\dagger})^!$, it extends to a map $A^![u]\to (A^{\dagger})^!$.  Since both sides are generated in degree 1 and the map is bijective in degree 1,
the map $A^![u]\to (A^{\dagger})^!$ is in fact a surjection. Since $u^2-w=0$ in $(A^{\dagger})^!$ after adjusting $w$ by a suitable scalar,
it induces a surjection $A^![u]/(u^2-w)\to (A^{\dagger})^!$.  By regrading $\deg A^!=0$ and $\deg u=1$ so that $u^2-w$ is a monic polynomial with respect to $u$,
we can check that $u^2-w\in A^![u]$ is a regular element.  Since $S, S^{\dagger}$ are quantum polynomial algebras, $A, A^{\dagger}$ are Koszul algebras, and $u^2-w\in A^![u]$ is a regular central element, 
it follows that
\begin{align*}
H_{A^![u]/(u^2-w)}(t) 
& =\dfrac{H_{A^!}(t)}{1-t}(1-t^2)=\dfrac{1+t}{H_A(-t)}=\dfrac{1+t}{H_S(-t)(1-t^2)}=\dfrac{(1+t)^n}{1-t} \\
H_{(A^{\dagger})^!}(t) & =\dfrac{1}{H_{A^{\dagger}}(-t)}=\dfrac{1}{H_{S^{\dagger}}(-t)(1-t^2)}=\dfrac{(1+t)^{n+1}}{1-t^2}=\dfrac{(1+t)^n}{1-t},
\end{align*}
so $A^![u]/(u^2-w)\cong (A^{\dagger})^!$.  Hence we obtain 
\begin{align*}
(A^{\dagger\dagger})^! & \cong (A^{\dagger})^![v]/(v^2-w) \cong (A^![u]/(u^2-w))[v]/(v^2-w) \\
& \cong A^![u, v]/(u^2-w, v^2-w) \cong A^![u, v]/(u^2-w, v^2-u^2)  \\
& \cong (A^![u]/(u^2-w))[v]/(v^2-u^2) \cong (A^{\dagger})^![v]/(v^2-u^2).
\end{align*}

(2) Since $(A^{\dagger})^!\cong A^![u]/(u^2-w)$ by (1), $w=u^2\in (A^{\dagger})^!_2$ is a regular central element such that 
$$(A^{\dagger})^!/(w)\cong A^![u]/(u^2, w)\cong S^![u]/(u^2)\cong (S^{\dagger})^!.$$
Similarly, since $(A^{\dagger\dagger})^!\cong A^![u, v]/(u^2-w, v^2-w)$ by (1), $w=u^2=v^2\in (A^{\dagger\dagger})^!_2$ is a regular central element such that 
$$(A^{\dagger\dagger})^!/(w)\cong A^![u, v]/(u^2, v^2, w)\cong S^![u, v]/(u^2, v^2)\cong (S^{\dagger\dagger})^!.$$

(3) By (1) and Lemma \ref{lem.nqh1} and Lemma \ref{lem.loce},  
\begin{align*}
C(A^{\dagger}):=(A^{\dagger})^![w^{-1}]_0 & \cong (A^{\dagger})^![(u^2)^{-1}]_0\cong (A^{\dagger})^![u^{-1}]_0, \\
C(A^{\dagger\dagger}):=(A^{\dagger\dagger})^![w^{-1}]_0 & \cong (A^{\dagger\dagger})^![(v^2)^{-1}]_0\cong (A^{\dagger\dagger})^![v^{-1}]_0 \cong (A^{\dagger\dagger})^![(u^2)^{-1}]_0\cong (A^{\dagger\dagger})^![u^{-1}]_0 \\
& \cong ((A^{\dagger})^![v]/(v^2-u^2))[u^{-1}]_0 \cong (A^{\dagger})^![u^{-1}]_0[v/u]/((v/u)^2-1) \\
& \cong C(A^{\dagger})\otimes _kk[t]/(t^2-1)\cong C(A^{\dagger})^{\times 2}.
\end{align*} 

\end{proof}   

We will now see that $C(A)$ is essential to compute $\uCM^{\ZZ}(A)$.
For a noetherian AS-Gorenstein algebra $A$, 
we define
$$\CM^{0}(A) := \{M \in \CM^{\ZZ}(A) \mid M \; \text{is generated in degree 0}\}.$$
The following lemma is a slight generalization of \cite[Lemma 5.1 (3), Lemma 5.1 (4), Proposition 5.2]{SV}.  

\begin{lemma} \label{lem.C(A)}
Let $A=S/(f)$ be a homogeneous coordinate ring of a quadric hypersurface in a quantum $\PP^{n-1}$.
Then
\begin{enumerate}
\item $\dim_k C(A)=2^{n-1}$. 
\item $F_{A^!}:\tails A^! \to \mod C(A)$ defined by $F_{A^!}(\pi M)=M[w^{-1}]_0$ is an equivalence functor.
\item $G:=\overline F_{(A^!)^o}\circ B:\uCM^{\ZZ}(A) \to \cD^b(\mod C(A)^o)$ is a duality where $\overline F_{(A^!)^o} :\cD^b(\tails (A^!)^o)\to \cD^b(\mod C(A)^o)$ is an equivalence functor induced by $F_{(A^!)^o}$.
\item $H:=D \circ G: \uCM^{\ZZ}(A) \to \cD^b(\mod C(A))$ is an equivalence functor.
\item $G$ restricts to a duality functor $\uCM^{0}(A) \to \mod C(A)^o$. 
\end{enumerate}
\end{lemma}

\begin{proof}
(1) 
Since $H_{S^!}(t)=(1+t)^n$,
$\dim_k C(A)= \dim_k (S^!)^{(2)} = 2^{n-1}$ by Lemma \ref{lem.dimc}.

(2)
Since $S^!=A^!/(w)$ is a graded Frobenius algebra, that is, a noetherian AS-Gorenstein algebra of dimension 0,
$A^!$ is a noetherian graded algebra generated in degree 1 over $k$ by Lemma \ref{lem.bal},
so the result follows from Proposition \ref{prop.sg}. 

(3) 
By Lemma \ref{lem.bal}, $A$ is a noetherian AS-Gorenstein Koszul algebra.
Since $S^!$ is a graded Frobenius algebra, $A^!$ is a noetherian Koszul AS-Gorenstein algebra by Lemma \ref{lem.bal} again, so the result follows from Lemma \ref{lem.Koz} and (2).  

(4) Since $D$  gives a duality between $\cD^b(\mod C(A)^o)$ and $\cD^b(\mod C(A))$,
this follows from (3).

(5) By \cite[Proposition 7.8 (1)]{MU}, $\CM^{0}(A)=\CM^{\ZZ}(A)\cap \lin A$, so the result follows from Lemma \ref{lem.Koz}.   
\end{proof}

Kn\"orrer's periodicity theorem is a powerful tool to compute $\uCM^{\ZZ}(A)$ since it reduces the number of variables.
The following theorem gives another way to reduce the number of variables in computing $\uCM^{\ZZ}(A)$. 

\begin{theorem} \label{thm.ca2} 
If $A=S/(f)$ is a homogeneous coordinate ring of a quadric hypersurface in a quantum $\PP^{n-1}$ where $f\in S_2$ is a regular central element, 
then $\uCM^{\ZZ}(A^{\dagger\dagger})\cong \uCM^{\ZZ}(A^{\dagger})\times \uCM^{\ZZ}(A^{\dagger})$.  
\end{theorem} 

\begin{proof} By Lemma \ref{lem.C(A)} (4) and Theorem \ref{thm.nqh2} (3), 
\begin{align*}
\uCM^{\ZZ}(A^{\dagger\dagger}) & \cong \cD^b(\mod C(A^{\dagger\dagger}))\cong \cD^b(\mod (C(A^{\dagger})\times C(A^{\dagger}))) \\
& \cong \cD^b(\mod C(A^{\dagger}))\times \cD^b(\mod C(A^{\dagger})) 
\cong \uCM^{\ZZ}(A^{\dagger})\times \uCM^{\ZZ}(A^{\dagger}).
 \end{align*}
 \end{proof}

\begin{example}  \label{ex.smrd}
If $S=k_{-1}[x_1, \dots, x_n]:=k\<x_1, \dots, x_n\>/(x_ix_j+x_jx_i)_{1\leq i, j\leq n, i\neq j}$, then $S$ is a quantum polynomial algebra of dimension $n$ and $f=x_1^2+\cdots +x_n^2\in S_2$ is a regular central element.
Since $S\cong k[x_1][x_2; -1]\cdots [x_n;-1]$,
we have $C(S/(f))\cong C(k[x_1]/(x_1^2))^{\times 2^{n-1}}\cong k^{2^{n-1}}$ by repeatedly applying Theorem \ref{thm.nqh2},
so $\uCM^{\ZZ}(S/(f))\cong \cD^b(\mod k^{2^{n-1}})$ (cf. \cite[Proposition 3.2, Theorem 3.3]{U}).
Since $f=(x_1+\e_2x_2+\cdots +\e_nx_n)^2$ for $\e_i=\pm 1$,  $f$ can be factorized in $2^{n-1}$ different ways, 
so all isomorphism classes of indecomposable non-free graded maximal Cohen-Macaulay modules over $A=S/(f)$ up to shifts are listed as $\Coker (x_1+\e_2x_2+\cdots +\e_nx_n)\cdot$ for $\e_i=\pm 1, i=2, \dots, n$.
(Note that, for every indecomposable non-free graded maximal Cohen-Macaulay modules $M$, there exists a noncommutative graded matrix factorization $\phi$ such that $\Coker \phi^0\cong M$ by \cite [Proposition 6.2]{MU}.)
\end{example}


\section{Noncommutative Smooth Quadric Hypersurfaces} 

From now on, we assume that $k$ is an algebraically closed field of characteristic not 2.
Recall that $A$ is the homogeneous coordinate ring of a smooth quadric hypersurface in $\PP^{n-1}$ if and only if $A\cong k[x_1, \dots, x_n]/(x_1^2+\cdots +x_n^2)$.
Applying the graded Kn\"orrer's periodicity theorem (Theorem \ref{thm.nkp}), we have 
$$\uCM^{\ZZ}(A)\cong \begin{cases} \uCM^{\ZZ}(k[x_1]/(x_1^2)) \cong  \cD^b(\mod k) & \text { if $n$ is odd,} \\
\uCM^{\ZZ}(k[x_1, x_2]/(x_1^2+x_2^2)) \cong  \cD^b(\mod (k \times k)) & \text { if $n$ is even.}\end{cases}$$
In this section, we prove a noncommutative analogue of this result up to $n=6$ under the high rank property.
In the next section, we will see that the high rank property is needed.    

\begin{definition}[Smoothness]
A graded algebra $A$ is called a homogeneous coordinate ring of a \emph{smooth} (resp. \emph{singular}) quadric hypersurface 
in a quantum $\PP^{n-1}$ 
if
\begin{itemize}
\item $A=S/(f)$ is a homogeneous coordinate ring of a quadric hypersurface in a quantum $\PP^{n-1}$, and
\item $\gldim(\tails A) < \infty$ (resp. $\gldim(\tails A) = \infty$)
\end{itemize}
\end{definition} 

We will give a characterization of the smoothness below.
We prepare the following two lemmas.  For a Krull-Schmidt additive category $\cC$, we denote by $\ind \cC$ a complete set of representatives of isomorphism classes of indecomposable objects of $\cC$. 

\begin{lemma} \label{lem.ind} 
Let $A=S/(f)$ be a homogeneous coordinate ring of a quadric hypersurface in a quantum $\PP^{n-1}$.
If $C(A)$ is semisimple, then, for every $M\in \CM^{\ZZ}(A)$, there exists $M_i\in \ind \CM^0(A)$, $\ell_i\in \ZZ$ and $r\in \NN$ such that $M\cong \bigoplus _{i=1}^rM_i(\ell_i)$.  
\end{lemma} 

\begin{proof} For $M\in \ind \uCM^{\ZZ}(A)$, we have
$G(M)\in \ind \cD^b(\mod C(A)^o)$. 
Since $C(A)$ is semisimple, there exists a simple module $N \in \mod C(A)^o$ such that $G(M) = N[\ell]$.
Since $G(\Omega^{-\ell}M) = G(M[\ell]) = G(M)[-\ell] = N\in \mod C(A)^o$, it follows that  $\Omega^{-\ell}M \in \uCM^{0}(A)$ by Lemma \ref{lem.C(A)} (5), so we have $M(\ell) \cong  \Omega ^{\ell}(\Omega^{-\ell}M)(\ell) \in \ind \uCM^{0}(A)$, hence the result.
\end{proof} 

A Serre functor $S$ for a $k$-linear Hom-finite additive category $\cT$ is a functor $S:\cT\to \cT$ such that $\Hom_{\cT}(Y, X)\cong \Hom_{\cT}(X, S(Y))^*$ for every $X, Y\in \cT$ (which are natural in both $X$ and $Y$) where $V^*$ denotes the dual vector space of $V$.
  
\begin{lemma} \label{lem.Serre} 
Let $\cT$ be a $k$-linear Hom-finite additive category.  If $\cT$ has a Serre functor, then so does $\cT^o$. 
\end{lemma} 

\begin{proof} Suppose that $\cT$ has a Serre functor $S$.  We write $X^o$ for an object $X\in \cT$ viewed as an object in $\cT^o$.  For $X^o, Y^o\in \cT^o$, 
\begin{align*}
\Hom_{\cT^o}(Y^o, X^o) & = \Hom_{\cT}(X, Y)\cong \Hom_{\cT}(Y, S(X))^*  \cong \Hom_{\cT}(S^{-1}(Y), X)^*\\
&=\Hom_{\cT^o}(X^o, S^{-1}(Y)^o)^*  \cong \Hom_{\cT^o}(X^o, S^{-1}(Y^o))^*,
\end{align*}
so $S^{-1}$ is a Serre functor for $\cT^o$.  
\end{proof} 


\begin{lemma} \label{lem.asSe} 
Let $A$ be a noetherian AS-Gorenstein algebra of dimension $d\geq 1$.
$\cD^b(\tails A)$ has a Serre functor if and only if $\gldim(\tails A)<\infty$, and in this case, $\gldim(\tails A)=d-1$.
\end{lemma}

\begin{proof}
This follows from \cite[Theorem A.4]{dNV}.
\end{proof}

\begin{theorem} \label{thm.smo}
Let $A=S/(f)$ be a homogeneous coordinate ring of a quadric hypersurface in a quantum $\PP^{n-1}$.
Then the following are equivalent:
\begin{enumerate}
\item $A$ is a homogeneous coordinate ring of a smooth quadric hypersurface in a quantum $\PP^{n-1}$.
\item $A$ has finite Cohen-Macaulay representation type (i.e., there exist only finitely many indecomposable graded maximal Cohen-Macaulay modules up to isomorphism and degree shifts).
\item $C(A)$ is semisimple.
\end{enumerate}
\end{theorem}

\begin{proof}
(1) $\Rightarrow$ (3): Since $\gldim(\tails A)<\infty$, $\uCM^{\ZZ}(A)$ has a Serre functor by \cite[Corollary 4.5]{Uis},
so $\cD^b(\tails (A^!)^o)$ has a Serre functor by Lemma \ref{lem.Koz} and Lemma \ref{lem.Serre}.
Since $(S^!)^o$ is noetherian AS-Gorenstein of dimension 0, $(A^!)^o$ is noetherian AS-Gorenstein of dimension 1 by Lemma \ref{lem.bal}, so $\gldim (\tails (A^!)^o) =0$ by Lemma \ref{lem.asSe}.
Since $\tails (A^!)^o \cong \mod C(A)^o$ by Lemma \ref{lem.C(A)} (2), we have $\gldim C(A)=\gldim (\mod C(A)^o)=0$.

(3) $\Rightarrow$ (2): 
Since $C(A)$ is semisimple, $|\ind\uCM^{0}(A)| = |\ind \mod C(A)^o|<\infty$ by Lemma \ref{lem.C(A)} (4), so the result follows from Lemma \ref{lem.ind}.

(2) $\Rightarrow$ (1): This follows from \cite[Theorem 3.4]{Ucm}.
\end{proof}

Let $S=k[x_1, \dots, x_n]$ and $0\neq f\in S_2$.  If $\Proj S/(f)$ is smooth, then $f$ is irreducible.  
However, this implication is not true in the noncommutative case.
In fact, for every $n$, there exists a homogeneous coordinate ring $S/(f)$ of a smooth quadric hypersurface in a quantum $\PP^{n-1}$ such that $f$ is reducible (see the next section),
so we will introduce the high rank property below.

\begin{definition}
Let $S$ be a graded algebra and let $0\neq f \in S_2$.  The \emph{rank} of $f$ is defined by
\[ \rank f := \min\{ r \in \NN^+ \mid f= \sum^{r}_{i=1}u_iv_i \;\textrm{where}\; u_i, v_i \in S_1 \}. \]
\end{definition} 
 
For $0 \neq f \in S_2$, we see that $f$ is irreducible if and only if $\rank f \geq 2$, so in this sense, $\rank f$ can be regarded as a generalization of irreducibility.
The following lemma is immediate. 

\begin{lemma} \label{lem.thra} 
Let $S$ be a graded algebra and $0\neq f\in S_2$.  
\begin{enumerate}
\item{} If $\varphi:S\to S'$ is an isomorphism of graded algebras, then $\rank f=\rank \varphi(f)$. 
\item{} If $\theta\in \Aut(S; f)$, then $\rank f=\rank f^{\theta}$. 
\end{enumerate}
\end{lemma} 

Let $S$ be a connected graded algebra and $f\in S$ a homogeneous regular normal element.  We say that 
$\phi\in \NMF_S^{\ZZ}(f)$ is \emph{reduced} if every entry in $\Phi^i$ is in $S_{\geq 1}$.  
Note that $\phi\in \NMF_S^{\ZZ}(f)$ is reduced if and only if $C(\phi)^{\geq 0}$ is the minimal free resolution of $\Coker \phi\in \grmod S/(f)$.

\begin{lemma} \label{lem.pfrank} 
Let $S$ be a connected graded algebra and $f\in S_2$ a homogeneous element.  For every reduced 
$\phi\in \NMF_S^{\ZZ}(f)$, we have $\rank \phi\geq \rank f$. 
\end{lemma} 

\begin{proof} Let $\rank \phi=r$.  Since $f\in S_2$ and $\phi\in \NMF_S^{\ZZ}(f)$ is reduced, every entry of $\Phi^i$ is in $S_1$.  Since $\Phi^i\Phi^{i+1}=fI_r$, the result follows. 
\end{proof} 

\begin{definition} [High rank property]
A graded algebra $A$ is called a homogeneous coordinate ring of a \emph{high rank} quadric hypersurface in a quantum $\PP^{n-1}$ 
if
\begin{itemize}
\item $A=S/(f)$ is a homogeneous coordinate ring of a quadric hypersurface in a quantum $\PP^{n-1}$, and
\item $f\in S_2$ is a regular normal element of
$
\rank f \geq
\begin{cases}
\frac{n+1}{2} &\textrm{if} \; n \; \text{is odd},\\
\frac{n}{2} &\textrm{if} \; n \; \text{is even}.
\end{cases}
$
\end{itemize}
\end{definition}

\begin{example}\label{ex.rank} 
Let $X$ be a quadric hypersurface in $\PP^{n-1}$.
If $X$ is smooth, then $X \cong \cV(x_1^2+\cdots +x_n^2)$ and
\begin{align*}
\rank (x_1^2+\cdots +x_n^2) =
\begin{cases}
\frac{n+1}{2} &\textrm{if} \; n \; \text{is odd},\\
\frac{n}{2} &\textrm{if} \; n \; \text{is even},
\end{cases}
\end{align*}
so $X$ is of high rank.  The converse holds for $n$ odd, but not for $n$ even.
For example, if $X=\cV(x_1^2+x_2^2+x_3^2)$ in $\PP^3$, then $\rank (x_1^2+x_2^2+x_3^2)= 2$, so $X$ is of high rank, but it is singular.
\end{example}

The following lemma is a slight generalization of \cite[Proposition 5.3 (2)]{SV}.

\begin{lemma} \label{lem.rank}
Let $A=S/(f)$ be a homogeneous coordinate ring of a quadric hypersurface in a quantum $\PP^{n-1}$.
Then $C(A)$ has no left modules of dimension less than $\rank f$.
\end{lemma}

\begin{proof} 
For $N\in \ind \mod C(A)^o$ such that $\dim _kN=r$, 
there exists an indecomposable non-projective graded maximal Cohen-Macaulay module 
$M$ generated in degree 0 such that $G(M)=N$ by Lemma \ref{lem.C(A)} (4).
Since $M\in \lin A$ by \cite[Theorem 3.2]{SV}, we have $E(M)[w^{-1}]_0 = N$.
We see that $\dim_k E(M)_{2i}=r$ for $i\gg 0$,
so the $2i$-th term of the minimal free resolution of $M$ is $A(-2i)^r$ for $i\gg 0$.
By \cite[Proposition 6.2]{MU}, this resolution is obtained from a reduced matrix factorization $\phi \in \NMF^{\ZZ}_S(f)$ of rank $r$.
By Lemma \ref{lem.pfrank}, $\dim_kN=r= \rank \phi \geq \rank f$.
\end{proof}

Since we work over an algebraically closed field $k$, the following lemma is immediate. 

\begin{lemma} \label{lem.lambda} 
Let $\Lambda$ be a semisimple algebra such that $\dim _k\Lambda =2^{n-1}$.
If 
$$\dim_kM> \begin{cases} 2^{(n-3)/2} & \textnormal { if $n$ is odd,} \\
 2^{(n-4)/2} & \textnormal { if $n$ is even,}\end{cases}$$ 
 for every $0\neq M\in \mod \Lambda^o$, then 
 $$\Lambda \cong \begin{cases} M_{2^{(n-1)/2}}(k) & \textnormal { if $n$ is odd,} \\
 M_{2^{(n-2)/2}}(k)\times M_{2^{(n-2)/2}}(k) & \textnormal { if $n$ is even.}\end{cases}$$
\end{lemma} 

\begin{lemma} \label{lem.scC(A)}
Let $A=S/(f)$ be a homogeneous coordinate ring of a smooth high rank quadric hypersurface in a quantum $\PP^{n-1}$.
\begin{enumerate}
\item If $n=1$, then $C(A)\cong k$.
\item If $n=2$, then $C(A)\cong k \times k$.
\item If $n=3$, then $C(A)\cong M_2(k)$.
\item If $n=4$, then $C(A)\cong M_2(k) \times M_2(k)$.
\item If $n=5$, then $C(A)\cong M_4(k)$.
\item If $n=6$, then $C(A)\cong M_4(k) \times M_4(k)$.
\end{enumerate}
\end{lemma}

\begin{proof} Since 
$$\rank f\geq \begin{cases} 
\frac{n+1}{2}=1>1/2=2^{(n-3)/2} & \textnormal { if $n=1$}, \\
\frac{n}{2}=1>1/2=2^{(n-4)/2}  & \textnormal { if $n=2$}, \\
\frac{n+1}{2}=2>1=2^{(n-3)/2} & \textnormal { if $n=3$}, \\
\frac{n}{2}=2>1=2^{(n-4)/2}  & \textnormal { if $n=4$}, \\
\frac{n+1}{2}=3>2=2^{(n-3)/2} & \textnormal { if $n=5$}, \\
\frac{n}{2}=3>2=2^{(n-4)/2}  & \textnormal { if $n=6$}, \\
\end{cases}$$
the result follows from Lemma \ref{lem.C(A)} (1), Lemma \ref{lem.rank}, and Lemma \ref{lem.lambda}. 
\end{proof} 

The following result can be considered as Kn\"orrer's periodicity theorem for homogeneous coordinate rings of a smooth high rank quadric hypersurface in a quantum $\PP^{n-1}$ with $n\leq 6$.
 
\begin{theorem} 
\label{thm.sKn1}
Let $A=S/(f)$ be a homogeneous coordinate ring of a smooth high rank quadric hypersurface in a quantum $\PP^{n-1}$, where $n \leq 6$. Then
$$\uCM^{\ZZ}(A)\cong
\begin{cases}
\cD^b(\mod k) \cong \uCM^{\ZZ}(k[x_1]/(x_1^2)) & \text { if $n$ is odd,} \\
\cD^b(\mod (k \times k)) \cong \uCM^{\ZZ}(k[x_1,x_2]/(x_1^2+x_2^2)) & \text { if $n$ is even.}
\end{cases}$$
\end{theorem}

\begin{proof} By Lemma \ref{lem.C(A)} (4) and Lemma \ref{lem.scC(A)}, 
if $n=1,3,5$, then 
\[ \uCM^{\ZZ}(A) \cong \cD^b(\mod C(A))  \cong  \cD^b(\mod k) \cong \uCM^{\ZZ}(k[x_1]/(x_1^2)), \]
and if $n=2,4,6$, then
\[ \uCM^{\ZZ}(A) \cong \cD^b(\mod C(A)) \cong \cD^b(\mod (k \times k)) \cong \uCM^{\ZZ}(k[x_1,x_2]/(x_1^2+x_2^2)).  \]
\end{proof}

\begin{remark} 
In the case that $A=S/(f)$ is a homogeneous coordinate ring of a smooth high rank quadric hypersurface in a quantum $\PP^{6}$ (i.e., $n=7$), $\rank f\geq \frac{n+1}{2}=4=2^{(n-3)/2}$, so we can only conclude that $C(A)\cong M_8(k)$ or $C(A)\cong M_4(k)\times M_4(k)\times M_4(k)\times M_4(k)$ by the above method.   
\end{remark} 

\section{$(\pm 1)$-Skew Polynomial algebras} 

In this section, we continue to assume that $k$ is an algebraically closed field of characteristic not 2.
Recall that $A$ is the homogeneous coordinate ring of a smooth quadric hypersurface in $\PP^{n-1}$ if and only if $A\cong k[x_1, \dots, x_n]/(x_1^2+\cdots +x_n^2)$.
In this section, we study a noncommutative analogue $A=S/(x_1^2+\cdots +x_n^2)$ where $S=k\<x_1, \dots, x_n\>/(x_ix_j-\e_{ij}x_jx_i)$, $\e_{ij}\in k, i\neq j$
such that $\e_{ij}\e_{ji}=1$, 
is a skew polynomial algebra, a typical example of a quantum polynomial algebra of dimension $n$.
Since $x_1^2+\cdots +x_n^2$ is a regular normal element in $S$ if and only if $\e_{ij}=\e_{ji}=\pm 1$, we assume that $\e_{ij}=\e_{ji}=\pm 1$ in this section.
In this case, $x_1^2+\cdots +x_n^2$ is a regular central element in $S$ and $A$ is a homogeneous coordinate ring of a quadric hypersurface in a quantum $\PP^{n-1}$.
We will classify $\uCM^{\ZZ}(A)$ for such $A$ up to $n=6$.
By the classification and Theorem \ref{thm.smo}, we will see that $A$ is a homogeneous coordinate ring of a smooth quadric hypersurface in a quantum $\PP^{n-1}$ if $n\leq 6$.     
The main method of computing $\uCM^{\ZZ}(A)$ is to use the graph associated to $S$.

A \emph{graph} $G$ consists of a finite set $V(G)$ of vertices and a set $E(G)$ of edges between two vertices.
In this section, we assume that every graph has no loop and there is at most one edge between two distinct vertices.
An edge between two vertices $v, w\in V(E)$ is written by $(v, w)\in E(G)$.  

\begin{definition} 
For $\e:=\{\e_{ij}\}_{1\leq i, j\leq n, i\neq j}$ where $\e_{ij}=\e_{ji}=\pm 1$, we fix the following notations: 
\begin{enumerate}
\item{} the graded algebra $S_{\e}:=k\<x_1, \dots, x_n\>/(x_ix_j-\e_{ij}x_jx_i)$, called a \emph{$(\pm 1)$-skew polynomial algebra} in $n$ variables, 
\item{} the point scheme $E_{\e}$ of $S_{\e}$, 
\item{} the central element $f_{\e}=x_1^2+\cdots +x_n^2\in S_{\e}$, 
\item{} the homogeneous coordinate ring $A_{\e}=S_{\e}/(f_{\e})$ of a quadric hypersurface in a quantum $\PP^{n-1}$, and 
\item{} the graph $G_{\e}$ where 
$V(G_{\e})=\{1, \dots , n\}$ and $E(G_{\e})=\{(i, j)\mid \e_{ij}=\e_{ji}=1\}$.  
\end{enumerate} 
\end{definition} 

If $\e_{ij}=1$ for all $i, j$, then $S_{\e}=k[x_1, \dots, x_n]$.  If $\e_{ij}=-1$ for all $i, j$, then $S_{\e}=k_{-1}[x_1, \dots, x_n]$.

We use the following lemma. 

\begin{lemma} \label{lem.ec} 
Let the notation be as above.
\begin{enumerate}
\item{} \textnormal{(cf. \cite[Theorem 2.3]{U})} $E_{\e}=\bigcap _{\e_{ij}\e_{jk}\e_{ki}=-1}\cV(x_ix_jx_k)\subset \PP^{n-1}$ . 
\item{} \textnormal{(\cite [Lemma 3.1]{U})} $C(A_{\e})\cong k\<t_1, \dots, t_{n-1}\>/(t_it_j+\e_{ni}\e_{ij}\e_{jn}t_jt_i, t_i^2-1)_{1\leq i, j\leq n-1, i\neq j}$. 
\end{enumerate}
\end{lemma}  

\subsection{Mutations} 

We introduce the notion of mutations, which preserve the stable category of graded maximal Cohen-Macaulay modules.  

\begin{definition} [Mutation]
Let $G$ be a graph and $v\in V(G)$. 
The \emph{mutation} $\mu_v(G)$ of $G$ at $v$ is a graph $\mu_v(G)$ where $V(\mu_v(G))=V(G)$ and 
$$E(\mu_v(G))=\{(v, u)\mid (v, u)\not \in E(G), u\neq v\}\cup \{(u, u')\mid (u, u')\in E(G), u, u'\neq v\}.$$
We say that graphs $G$ and $G'$ are \emph{mutation equivalent}
if $G'$ is isomorphic to a graph obtained from $G$ by applying mutations finite number of times. 
\end{definition} 

It is easy to see that the property of being mutation equivalent is an equivalence relation.

\begin{lemma} \label{lem.I}  
For a subset $I\subset \{1, \dots, n\}$, there exists $\theta_I\in \GrAut (S_{\e}; f_{\e})$ defined by $\theta_I(x_i)=\begin{cases} -x_i & \textnormal { if } i\in I, \\ x_i & \textnormal { if } i\not \in I. \end{cases}$  Moreover, $(S_{\e})^{\theta_I}=S_{\e'}$ for some $\e'$, and $\NMF^{\ZZ}_{S_{\e}}(f_{\e})\cong \NMF^{\ZZ}_{S_{\e'}}(f_{\e'})$.  
\end{lemma} 

\begin{proof} It is clear that $\theta_I\in \GrAut (S_{\e}; f_{\e})$ and $(S_{\e})^{\theta_I}=S_{\e'}$ for some $\e'$.
Since there exists $\sqrt{ \theta_I}\in \GrAut S_{\e'}$ defined by $\sqrt {\theta_I}(x_i)=\begin{cases} \sqrt{-1}x_i & \textnormal { if } i\in I, \\ x_i & \textnormal { if } i\not \in I \end{cases}$ such that $\sqrt{\theta_I}((f_{\e})^{\theta_I})=f_{\e'}$, 
$$\NMF^{\ZZ}_{S_{\e}}(f_{\e})\cong \NMF^{\ZZ}_{(S_{\e})^{\theta_I}}((f_{\e})^{\theta_I})\cong \NMF^{\ZZ}_{S_{\e'}}(f_{\e'})$$
by Lemma \ref{lem.eq2} and Lemma \ref{lem.eq1}.   
\end{proof} 

\begin{lemma} [Mutation Lemma] \label{lem.mg}  \label{lem.ramu} 
If $G_{\e}$ and $G_{\e'}$ are mutation equivalent, then the following hold: 
\begin{enumerate}
\item{} $\GrMod S_{\e}\cong \GrMod S_{\e'}$. 
\item{} $\NMF_{S_{\e}}^{\ZZ}(f_{\e})\cong \NMF_{S_{\e'}}^{\ZZ}(f_{\e'})$.
\item{} $E_{\e}\cong E_{\e'}$.  
\item{} $C(A_{\e})\cong C(A_{\e'})$.
\item{} $\uCM^{\ZZ}(A_{\e})\cong \uCM^{\ZZ}(A_{\e'})$.  
\item{} $\rank f_{\e}=\rank f_{\e'}$. 
\end{enumerate}  
\end{lemma} 

\begin{proof} By reordering the vertices and the iteration, it is enough to show the case $\mu_n(G_{\e})=G_{\e'}$. 

(1) and (2) If $I=\{n\}$ and $\theta=\theta_I\in \Aut (S; f)$ as defined in Lemma \ref{lem.I},
then $(S_{\e})^{\theta}\cong S_{\e'}$, so $\GrMod S_{\e}\cong \GrMod S_{\e'}$ by Lemma \ref{lem.ztw}, and $\NMF_{S_{\e}}^{\ZZ}(f_{\e})\cong \NMF_{S_{\e'}}^{\ZZ}(f_{\e'})$ by Lemma \ref{lem.I}.   

(3) and (4) Since $\e'_{ij}\e'_{jk}\e'_{ki}=\e_{ij}\e_{jk}\e_{ki}$ for every $i<j<k$, they follow from Lemma \ref{lem.ec}  (1) and (2), respectively.   

(5) This follows from (4) and Lemma \ref{lem.C(A)} (4).  

(6) This follows from the proof of Lemma \ref{lem.I} and Lemma \ref{lem.thra}. 
\end{proof}

\begin{definition}[Relative Mutation]
Let $v, w \in V(G)$ be distinct vertices.
Then the \emph{relative mutation} $\mu_{v \leftarrow w}(G)$ of $G$ at $v$ by $w$ is a graph $\mu_{v \leftarrow w}(G)$
where $V(\mu_{v \leftarrow w}(G))=V(G)$ and 
$E(\mu_{v \leftarrow w}(G))$ is given by the following rules:
\begin{enumerate}
\item
For distinct vertices $u, u' \neq v$, we define that $(u, u') \in E(\mu_{v \leftarrow w}(G))$ $:\Leftrightarrow$ $(u,u') \in E(G)$.
\item
For a vertex $u \neq v, w$, we define that 
$$(v,u) \in E(\mu_{v \leftarrow w}(G)) :\Leftrightarrow
\begin{cases}
(v,u) \in E(G) \;\textrm{and}\; (w,u) \not\in E(G), \textrm{or}\\
(v,u) \not\in E(G) \;\textrm{and}\; (w,u) \in E(G).
\end{cases}$$
\item
We define that $(v,w) \in E(\mu_{v \leftarrow w}(G))$ $:\Leftrightarrow$ $(v,w) \in E(G)$.
\end{enumerate}
\end{definition}

\begin{lemma}[Relative Mutation Lemma]\label{lem.M2}
Let $u, v, w\in V(G_{\e})$ be distinct vertices.  
If $u$ is an isolated vertex, and $G_{\e'}=\mu_{v \leftarrow w}(G_\e)$, then $C(A_{\e})\cong C(A_{\e'})$ and hence $\uCM^{\ZZ}(A_{\e})\cong \uCM^{\ZZ}(A_{\e'})$.
\end{lemma}

\begin{proof} By reordering the vertices, we may assume that $u=n, v=n-1, w=n-2$.  
By Lemma \ref{lem.ec} (2), 
$C(A_{\e}) \cong k\<t_1, \dots, t_{n-1}\>/(t_it_j +\e_{ni} \e_{ij} \e_{jn} t_jt_i,\; t_i^2-1)$.
Let $\Lambda$ be the algebra generated by $s_1, \dots, s_{n-1}$ with defining relations
\begin{align*}
&s_is_j+\e_{ni}\e_{ij}\e_{jn}s_js_i \quad (1\leq i,j\leq n-2, i \neq j) ,\\
&s_{n-1}s_j -\e_{n,n-1}\e_{n-1,j}\e_{jn}\e_{n,n-2}\e_{n-2,j}\e_{jn}s_js_{n-1} \quad (1\leq j\leq n-3), \\
&s_{n-1}s_{n-2} +\e_{n,n-1}\e_{n-1,n-2}\e_{n-2,n} s_{n-2}s_{n-1},\\
&s_i^2-1 \quad (1\leq i\leq n-2), \\
&s_{n-1}^2+\e_{n,n-1}\e_{n-1,n-2}\e_{n-2,n}.
\end{align*}
Define a map
\begin{align*}
\widetilde \phi :\ & k\<s_1, \dots, s_{n-1}\> \to k\<t_1, \dots, t_{n-1}\>/(t_it_j +\e_{ni}\e_{ij} \e_{jn} t_jt_i,\; t_i^2-1);\\ 
& s_i \mapsto t_i \;\; (1\leq i \leq n-2),\quad  s_{n-1} \mapsto  t_{n-1}t_{n-2}.
\end{align*} 
Then one can check that $\widetilde \phi$ sends every defining relation of $\Lambda$ to zero,
so we have an induced map $\phi : \Lambda \to k\<t_1, \dots, t_{n-1}\>/(t_it_j +\e_{ni}\e_{ij} \e_{jn} t_jt_i,\; t_i^2-1)$.
It is easy to see that $\phi$ is an isomorphism.

By definition of $G_{\e'}=\mu_{n-1 \leftarrow n-2}(G_\e)$, we have $\e_{ni}\e_{ij}\e_{jn}=\e_{ni}'\e_{ij}'\e_{jn}'$ for $1\leq i,j\leq n-2,i \neq j$.
Moreover, since $\e_{in}=-1$ for any $1\leq i\leq n-1$, we have
\begin{align*}
-\e_{n,n-1}\e_{n-1,j}\e_{jn}\e_{n,n-2}\e_{n-2,j}\e_{jn}
&=(\e_{jn}\e_{n,n-1}\e_{n,n-2})(-\e_{n-1,j}\e_{n-2,j})\e_{jn}\\
&=(-\e_{n,n-1}\e_{n,n-2})(-\e_{n-1,j}\e_{n-2,j})\e_{jn}\\
&=\e_{n,n-1}'\e_{n-1,j}'\e_{jn}'
\end{align*}
for $1\leq j\leq n-3$, and 
\begin{align*}
\e_{n,n-1}\e_{n-1,n-2}\e_{n-2,n}&=(-\e_{n,n-1}\e_{n,n-2})\e_{n-1,n-2}\e_{n-2,n}=\e_{n,n-1}'\e_{n-1,n-2}'\e_{n-2, n}'.
\end{align*}
Hence it follows that $\Lambda$ is isomorphic to $C(A_{\e'})$. 
\end{proof}

\begin{example} \label{ex.M2}
By Lemma \ref{lem.M2} and Lemma \ref{lem.ramu}, $\uCM^{\ZZ}(A)$ is preserved by the following operations on the graph: 
\begin{enumerate}
\item 
\[ G= \xy /r2pc/: 
{\xypolygon5{~={90}~*{\xypolynode}~>{}}},
"1";"2"**@{-},
"2";"3"**@{-},
"3";"4"**@{-},
\endxy
\quad \Longrightarrow \quad 
\mu_{1 \leftarrow 2}(G) = \xy /r2pc/: 
{\xypolygon5{~={90}~*{\xypolynode}~>{}}},
"1";"2"**@{-},
"1";"3"**@{-},
"2";"3"**@{-},
"3";"4"**@{-},
\endxy
\quad 
\Longrightarrow \quad
\mu_3(\mu_{1 \leftarrow 2}(G)) = \xy /r2pc/: 
{\xypolygon5{~={90}~*{\xypolynode}~>{}}},
"1";"2"**@{-},
"3";"5"**@{-},
\endxy.
\]
\item
\[ G= \xy /r2pc/: 
{\xypolygon6{~={90}~*{\xypolynode}~>{}}},
"1";"2"**@{-},
"2";"3"**@{-},
"3";"4"**@{-},
"4";"5"**@{-},
\endxy
\quad \Longrightarrow \quad 
\mu_{1 \leftarrow 2}(G) = \xy /r2pc/: 
{\xypolygon6{~={90}~*{\xypolynode}~>{}}},
"1";"2"**@{-},
"1";"3"**@{-},
"2";"3"**@{-},
"3";"4"**@{-},
"4";"5"**@{-},
\endxy
\quad 
\Longrightarrow \quad
\mu_3(\mu_{1 \leftarrow 2}(G)) = \xy /r2pc/: 
{\xypolygon6{~={90}~*{\xypolynode}~>{}}},
"1";"2"**@{-},
"3";"5"**@{-},
"3";"6"**@{-},
"4";"5"**@{-},
\endxy.
\]
\item
\[ G= \xy /r2pc/: 
{\xypolygon6{~={90}~*{\xypolynode}~>{}}},
"1";"2"**@{-},
"1";"5"**@{-},
"2";"3"**@{-},
"3";"4"**@{-},
"4";"5"**@{-},
\endxy
\quad \Longrightarrow \quad 
\mu_{1 \leftarrow 4}(G) = \xy /r2pc/: 
{\xypolygon6{~={90}~*{\xypolynode}~>{}}},
"1";"2"**@{-},
"1";"3"**@{-},
"2";"3"**@{-},
"3";"4"**@{-},
"4";"5"**@{-},
\endxy
\quad 
\Longrightarrow\quad
\mu_3(\mu_{1 \leftarrow 4}(G)) = \xy /r2pc/: 
{\xypolygon6{~={90}~*{\xypolynode}~>{}}},
"1";"2"**@{-},
"3";"5"**@{-},
"3";"6"**@{-},
"4";"5"**@{-},
\endxy.
\]
\end{enumerate}
\end{example}

\subsection{Rank} 

\begin{lemma} \label{lem.ra1}
For $\e:=\{\e_{ij}\}_{1\leq i, j\leq n, i\neq j}$ where $\e_{ij}=\e_{ji}=\pm 1$, the following are equivalent: 
\begin{enumerate}
\item{} $\e_{ij}\e_{jk}\e_{ki}=-1$ for every $1\leq i<j<k\leq n$. 
\item{} $\dim E_{\e}=1$ (i.e., $E_{\e}$ has the lowest dimension). 
\item{} $C(A_{\e})\cong k^{2^{n-1}}$.  
\item{} $\uNMF^{\ZZ}_{S_{\e}}(f_{\e})\cong \uCM^{\ZZ}(A_{\e})\cong \cD^b(\mod k^{2^{n-1}})$.  
\item{} $\rank f_{\e}=1$ (i.e., $f_{\e}$ has the lowest rank). 
\end{enumerate}
\end{lemma} 

\begin{proof} By \cite [Proposition 3.2, Theorem 3.3]{U}, we see (1) $\Leftrightarrow$ (2) $\Leftrightarrow$ (3) $\Leftrightarrow$ (4).  

(3) $\Rightarrow $ (5):
If $C(A_{\e})\cong k^{2^{n-1}}$, then there exists $N\in \mod C(A)$ such that $\dim _kN=1$.  As in the proof of Lemma \ref{lem.rank}, there exists a reduced matrix factorization $\phi\in \NMF^{\ZZ}_{S_{\e}}(f_{\e})$ such that $\rank \phi=1$.  Since $\rank f_{\e}\leq \rank \phi$ by Lemma \ref{lem.pfrank}, we have $\rank f_{\e}=1$.  

(5) $\Rightarrow $ (1): If $\rank f_{\e}=1$, then 
\begin{align*}
f_{\e} 
& =x_1^2+\cdots +x_n^2 \\
& =(\a_1x_1+\cdots +\a_nx_n)(\b_1x_1+\cdots +\b_nx_n) \\
& =\sum_{i=1}^n\a_i\b_ix_i^2+\sum_{1\leq i<j\leq n}(\a_i\b_jx_ix_j+\a_j\b_ix_jx_i).  
\end{align*}
It follows that $\a_i\b_i=1$ for every $1\leq i\leq n$ and $\e_{ij}=-\a_j\b_i/\a_i\b_j$ for every $1\leq i, j\leq n$, so
$$\e_{ij}\e_{jk}\e_{ki}=(-\a_j\b_i/\a_i\b_j)(-\a_k\b_j/\a_j\b_k)(-\a_i\b_k/\a_k\b_i)=-1$$
for every $1\leq i<j<k\leq n$.  
\end{proof}

\begin{definition}  
Let $G$ be a graph.
A graph $G'$ is a \emph{full subgraph} of $G$ if $V(G')\subset V(G)$ and $E(G')=\{(v, w)\in E(G)\mid v, w\in V(G')\}$.
For a subset $I\subset V(G)$, we denote by $G\setminus I$ the full subgraph of $G$ such that $V(G\setminus I)=V(G)\setminus I$.  For a full subgraph $G'$ of $G$, we define the \emph{complement graph} of $G'$ in $G$ by $G\setminus G':=G\setminus V(G')$. 
\end{definition} 

\begin{lemma} \label{lem.reo}
If $n$ is even, then $\rank f_{\e}\leq \frac {n}{2}$.   If $n$ is odd, then $\rank f_{\e}\leq \frac {n+1}{2}$.  
\end{lemma} 

\begin{proof} First, note that if $n=2$, that is, if $f_{\e}=x_1^2+x_2^2\in S_{\e}=k_{\pm 1}[x_1, x_2]$, then $\rank f_{\e}=1$ for every $\e$.   

Suppose $n$ is even. If $G_{\e_i}$ is the full subgraph of $G_{\e}$ such that $V(G_{\e_i})=\{i, i+\frac{n}{2}\}$ for $i=1, \dots, \frac{n}{2}$, then $\rank f_{\e}\leq \rank f_{\e_1}+\cdots +\rank f_{\e_{n/2}}=\frac {n}{2}$.  

Suppose $n$ is odd.  If $G_{\e_i}$ is the full subgraph of $G_{\e}$ such that $V(G_{\e_i})=\{i, i+\frac{n-1}{2}\}$ for $i=1, \dots, \frac{n-1}{2}$, and $G_{\e_{(n+1)/2}}$ is the full subgraph of $G_{\e}$ such that $V(G_{\e_{(n+1)/2}})=\{n\}$, then $\rank f_{\e}\leq \rank f_{\e_1}+\cdots +\rank f_{\e_{(n+1)/2}}=\frac {n+1}{2}$. 
\end{proof} 

\begin{proposition} If $\dim E_{\e}=n-1$ (i.e., $E_{\e}=\PP^{n-1}$ has the highest dimension), 
then $A_{\e}$ is a homogeneous coordinate ring of a high rank quadric hypersurface in a quantum $\PP^{n-1}$ (i.e., $f_{\e}$ has the highest rank).  If $n$ is odd, then the converse also holds.    
\end{proposition} 

\begin{proof} By Lemma \ref{lem.ec} (1), $E_{\e}=\PP^{n-1}$ if and only if $\e_{ij}\e_{jk}\e_{ki}=1$ for every $i, j, k\in V(G_{\e})$.  

Suppose that $E_{\e}=\PP^{n-1}$.
By mutating at all the vertices of a proper connected component of $G_{\e}$ if exists, we may assume that $G_{\e}$ is a connected graph by Lemma \ref{lem.ramu} (see Example \ref{ex.cg}).
The condition $\e_{ij}\e_{jk}\e_{ki}=1$ for every $i, j, k\in V(G_{\e})$ means if $(i, j), (j, k)\in E(G_{\e})$, then $(i, k)\in E(G_{\e})$.  By induction, if $(i_1, i_2), (i_2, i_3), \dots, (i_{m-1}, i_m)\in E(G_{\e})$, then $(i_1, i_m)\in E(G_{\e})$.  Since $G_{\e}$ is connected, $(i, j)\in E(G_{\e})$ for every $i, j\in V(G_{\e})$, so $\e_{ij}=1$ for every $1\leq i<j\leq n$.  It follows that $S_{\e}=k[x_1, \dots, x_n]$, hence $A_{\e}$ is a homogeneous coordinate ring of a high rank quadric hypersurface in a quantum $\PP^{n-1}$ (see Example \ref{ex.rank}). 

Conversely, suppose that $n$ is odd.  If $E_{\e}\neq \PP^{n-1}$, then there exist $1\leq i<j<k\leq n$ such that $\e_{ij}\e_{jk}\e_{ki}=-1$.  By Lemma \ref{lem.ramu}, we may assume that $\e_{n-2, n-1}=\e_{n-1, n}=\e_{n, n-2}=-1$ by mutation and reordering so that there exists a full subgraph $G'$ of $G$ consisting of three isolated points $n-2, n-1, n$.  If $G_{\e'}=G\setminus G'$, then $f_{\e}=f_{\e'}+(x_{n-2}+x_{n-1}+x_n)^2$, so $\rank f_{\e}\leq \rank f_{\e'}+1\leq \frac{n-3}{2}+1=\frac{n-1}{2}$ by Lemma \ref{lem.reo}, hence $A_{\e}$ does not satisfy the high rank property.  
\end{proof} 

\begin{example} \label{ex.cg}
We can always transform $G_{\e}$ to be a connected graph by mutations.
For example, if $G_\e$ is as follows, then the vertices $1,2$ determine a proper connected component of $G_\e$, and 
$\mu_2\mu_1 (G_\e)$ is a connected graph:
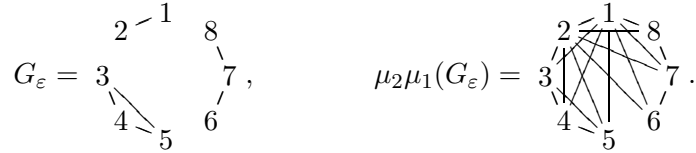
\[ G_\e = \xy /r2pc/: 
{\xypolygon8{~={90}~*{\xypolynode}~>{}}},
"1";"2"**@{-},
"3";"4"**@{-},
"4";"5"**@{-},
"5";"3"**@{-},
"6";"7"**@{-},
"7";"8"**@{-},
\endxy,
\qquad \qquad 
\mu_2\mu_1 (G_\e)= \xy /r2pc/: 
{\xypolygon8{~={90}~*{\xypolynode}~>{}}},
"1";"2"**@{-},
"1";"3"**@{-},
"1";"4"**@{-},
"1";"5"**@{-},
"1";"6"**@{-},
"1";"7"**@{-},
"1";"8"**@{-},
"2";"3"**@{-},
"2";"4"**@{-},
"2";"5"**@{-},
"2";"6"**@{-},
"2";"7"**@{-},
"2";"8"**@{-},
"3";"4"**@{-},
"4";"5"**@{-},
"5";"3"**@{-},
"6";"7"**@{-},
"7";"8"**@{-},
\endxy.
\]

\end{example} 

\begin{example} \label{ex.n3} 
If $n=3$, then we have the following equivalences: 
\begin{enumerate}
\item{} $\rank f_{\e}=2$ \; $\Leftrightarrow$ \; $E_{\e}=\PP^2$ \; $\Leftrightarrow$ \; $\uCM^{\ZZ}(A_{\e})\cong \cD^b(\mod k)$. 
\item{} $\rank f_{\e}=1$ \; $\Leftrightarrow$ \; $E_{\e}\cong \cV(xyz)\subset \PP^2$ \; $\Leftrightarrow$ \; $\uCM^{\ZZ}(A_{\e})\cong \cD^b(\mod k^4)$. 
\end{enumerate}

In fact, let $S=k\<x, y, z\>/(yz-\a zy, zx-\b xz, xy-\c yx)$ be a $(\pm 1)$-skew polynomial algebra and $f=x^2+y^2+z^2\in S$.
Note that there are essentially four options for $S$, corresponding to how many of $\a, \b,\c$ are $-1$.
Now we have $\theta\in \Aut _0(S; f)$ defined by $\theta(x)=x, \theta(y)=-y, \theta(z)=-z$. 
Since 
\begin{align*}
& k[x, y, z]^{\theta}\cong k\<x, y, z\>/(yz-zy, zx+xz, xy+yx)\cong k[y, z][x; -1], \\
& k_{-1}[x, y, z]^{\theta}\cong k\<x, y, z\>/(yz+zy, zx-xz, xy-yx)\cong (k_{-1}[y, z])[x], 
\end{align*}
there are two possibilities for the stable categories, namely,
either 
\begin{enumerate}
\item{} $\GrMod S/(f)$ is equivalent to $\GrMod k[x, y, z]/(x^2+y^2+z^2)$ so that $\rank f=2$, the point scheme of $S$ is $E=\PP^2$, and 
$\uCM^{\ZZ}(S/(f))\cong \uCM^{\ZZ}(k[x, y, z]/(x^2+y^2+z^2))\cong \cD^b(\mod k)$, or 
\item{} $\GrMod S/(f)$ is equivalent to $\GrMod k_{-1}[x, y, z]/(x^2+y^2+z^2)$ so that $\rank f=1$, the point scheme of $S$ is $E=\cV(xyz)\subset \PP^2$, and 
$\uCM^{\ZZ}(S/(f))\cong \uCM^{\ZZ}(k[x, y, z]/(x^2+y^2+z^2))\cong \cD^b(\mod k^4)$.  
\end{enumerate}

If $S=k[x, y, z]$, then 
$$\begin{pmatrix} x & y+\sqrt{-1}z \\ y-\sqrt{-1}z & -x \end{pmatrix}^2=\begin{pmatrix} f & 0 \\ 0 & f \end{pmatrix},$$
so all isomorphism classes of indecomposable non-free maximal graded Cohen-Macaulay modules over $A=S/(f)$ up to shifts are listed as 
$$\Coker \begin{pmatrix} x & y+\sqrt{-1}z \\ y-\sqrt{-1}z & -x \end{pmatrix}\cdot ,$$
whose corresponding noncommutative graded right matrix factorizations are of rank $2$.

If $S=k_{-1}[x, y, z]$, then 
$$(x+y+z)^2=(x-y+z)^2=(x+y-z)^2=(x-y-z)^2=f,$$
so all isomorphism classes of indecomposable non-free graded maximal Cohen-Macaulay modules over $A=S/(f)$ up to shifts are listed as 
$$\Coker (x+y+z)\cdot,\; \Coker (x-y+z)\cdot,\; \Coker (x+y-z)\cdot,\; \Coker(x-y-z)\cdot,$$
whose corresponding noncommutative graded right matrix factorizations are of rank $1$ (see Example \ref{ex.smrd}).
\end{example} 

\subsection{Reductions} 

We will give two ways to reduce the number of variables in computing $\uCM^{\ZZ}(A)$.  One is coming from the noncommutative Kn\"orrer's periodicity theorem (Theorem \ref{thm.nkp}), and the other is coming from Theorem \ref{thm.ca2}. 

\begin{theorem} \label{thm.nkp2} 
If $S$
is a ($\pm 1$)-skew polynomial algebra in $n$ variables and $f=x_1^2+\cdots +x_n^2\in S$, then 
$\uCM^{\ZZ}(S[u, v]/(f+u^2+v^2))\cong \uCM^{\ZZ}(S/(f))$.  
\end{theorem}

\begin{proof} 
This is a special case of Theorem \ref{thm.nkp}. 
\end{proof}

\begin{definition}
An \emph{isolated segment} $[v, w]$ of a graph $G$ consists of distinct vertices $v, w\in V(G)$ with an edge $(v, w)\in E(G)$ between them such that neither $v$ nor $w$ are connected by an edge to any other vertex.  
\end{definition}

\begin{lemma} [Kn\"orrer's Reduction] \label{lem.R2} 
Suppose that $[i, j]$ is an isolated segment in $G_{\e}$. If $G_{\e'}=G_{\e}\setminus \{i, j\}$, 
then 
$\uCM^{\ZZ}(A_{\e})\cong \uCM^{\ZZ}(A_{\e'})$. 
\end{lemma} 

\begin{proof}  By reordering the vertices, we may assume that $i=n, j=n-1$.
If $G_{\e''}=\mu_{n-1}\mu_n(G_{\e})$, then $x_{n-1}, x_n$ are central elements in $S_{\e''}$, so we can apply Theorem \ref{thm.nkp2} to delete $x_{n-1}, x_n$ from $S_{\e''}$.
By Lemma \ref{lem.mg},
$$\uCM^{\ZZ}(A_{\e})\cong \uCM^{\ZZ}(A_{\e''})\cong \uCM^{\ZZ}(A_{\e'}).$$ 
\end{proof} 

The following lemma also reduces the number of variables in computing $\uCM^{\ZZ}(A_{\e})$.  

\begin{lemma} [Two Point Reduction] \label{lem.R1} 
Suppose that $i, j\in V(G_{\e})$ are two distinct isolated vertices.  If $G_{\e'}=G_{\e}\setminus \{i\}$, then $C(A_{\e})\cong C(A_{\e'}) \times C(A_{\e'})$ and $\uCM^{\ZZ}(A_{\e})\cong \uCM^{\ZZ}(A_{\e'}) \times \uCM^{\ZZ}(A_{\e'})$.
\end{lemma} 

\begin{proof}
By reordering the vertices, we may assume that $i=n, j=n-1$.  If $A=A_{\e}/(x_{n-1}, x_n)$, then $A$ is a homogeneous coordinate ring of a quadric hypersurface in a quantum $\PP^{n-3}$ such that $A_{\e'}\cong A^{\dagger}$ and $A_{\e}\cong A^{\dagger\dagger}$, so the result follows from Theorem \ref{thm.nqh2} and Theorem \ref{thm.ca2}.
\end{proof}

\subsection{The Classification up to $n=6$} \label{ss.class}

We now classify $G_{\e}$, $E_{\e}$, $\uCM^{\ZZ}(A_{\e})$ for $n=4, 5, 6$ by the following steps (see Example \ref{ex.n3} for the case $n=3$): 

\begin{enumerate}
\item[(\Rnum 1)] Classify $G_{\e}$ up to mutation equivalence. 
\item[(\Rnum 2)] Compute $E_{\e}$ from each graph $G_{\e}$.
\item[(\Rnum 3)] Compute $\uCM^{\ZZ}(A_{\e})$ from each graph $G_{\e}$
by using Two Point Reduction (Lemma \ref{lem.R1}), Kn\"orrer's Reduction (Lemma \ref{lem.R2}) and Relative Mutation Lemma (Lemma \ref{lem.M2}).
\end{enumerate} 

In step (\Rnum 1), we may first assume that every vertex of $G_{\e}$ is of degree 0 or 1 if $n=4$, and is of degree 0, 1, 2 if $n=5, 6$, so we can start with a reasonably short list of graphs.

\begin{remark} The classifications of  $E_{\e}$ and $\uCM^{\ZZ}(A_{\e})$ for $n=4, 5$ were already completed in \cite[Theorem 3.9]{U}.  We repeat the classifications since we claim that our new reduction techniques  (Lemma \ref{lem.R2}, Lemma \ref{lem.R1}) developed in this paper are more effective.  In fact, we can complete the classification even in the case $n=6$ by reducing the number of variables.   
\end{remark} 

\subsubsection{The case $n=4$}

(\Rnum 1) There are exactly three graphs
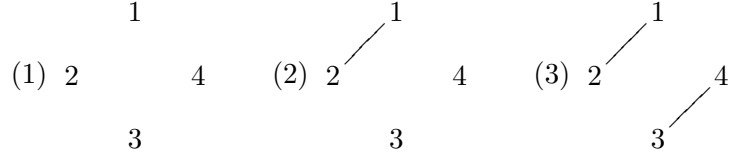
\[(1)\; \xy /r2pc/: 
{\xypolygon4{~={90}~*{\xypolynode}~>{}}},
\endxy
\qquad
(2)\; \xy /r2pc/: 
{\xypolygon4{~={90}~*{\xypolynode}~>{}}},
"1";"2"**@{-},
\endxy
\qquad
(3)\; \xy /r2pc/: 
{\xypolygon4{~={90}~*{\xypolynode}~>{}}},
"1";"2"**@{-},
"3";"4"**@{-},
\endxy
\]
up to mutation equivalence.

(\Rnum 2) There are exactly three point schemes
\begin{enumerate}
\item $\bigcup_{1\leq i<j\leq 4} \cV(x_i, x_j)$ \quad [$\ell=6$]
\item $\cV(x_1, x_2)\cup \cV(x_3)\cup \cV(x_4)$ \quad [$\ell=1$]
\item $\PP^3$ \quad [$\ell=0$]
\end{enumerate}
where $\ell$ is the number of irreducible components that are isomorphic to $\PP^1$.

(\Rnum 3) In the case (1), we can apply Lemma \ref{lem.R1} twice to obtain 
$$C(A_{\e})\cong C(k_{-1}[x_1, x_2]/(x_1^2+x_2^2))^{\times 4}\cong k^8,$$ 
so we have $\uCM^{\ZZ}(A_{\e})\cong \cD^b(\mod k^8)$.
In the case (2), we can apply Lemma \ref{lem.R2} to obtain $\uCM^{\ZZ}(A_{\e})\cong \uCM^{\ZZ}(k_{-1}[x_1, x_2]/(x_1^2+x_2^2))\cong \cD^b(\mod k^2)$.  In the case (3), we can apply Lemma \ref{lem.R2} to obtain $\uCM^{\ZZ}(A_{\e})\cong \uCM^{\ZZ}(k[x_1, x_2]/(x_1^2+x_2^2))\cong \cD^b(\mod k^2)$.

\subsubsection{The case $n=5$}
(\Rnum 1) There are exactly seven graphs
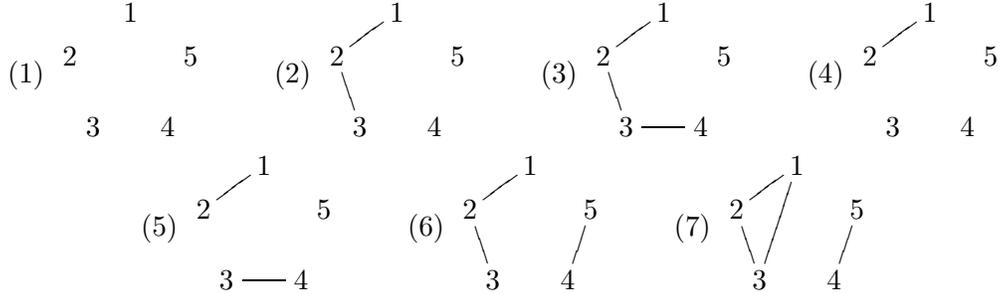
\begin{center}
(1) $\xy /r2pc/: 
{\xypolygon5{~={90}~*{\xypolynode}~>{}}},
\endxy $
\qquad 
(2) $\xy /r2pc/: 
{\xypolygon5{~={90}~*{\xypolynode}~>{}}},
"1";"2"**@{-},
"2";"3"**@{-},
\endxy $
\qquad 
(3) $\xy /r2pc/: 
{\xypolygon5{~={90}~*{\xypolynode}~>{}}},
"1";"2"**@{-},
"2";"3"**@{-},
"3";"4"**@{-},
\endxy $
\qquad 
(4) $\xy /r2pc/: 
{\xypolygon5{~={90}~*{\xypolynode}~>{}}},
"1";"2"**@{-},
\endxy$

(5) $\xy /r2pc/: 
{\xypolygon5{~={90}~*{\xypolynode}~>{}}},
"1";"2"**@{-},
"3";"4"**@{-},
\endxy $
\qquad 
(6) $\xy /r2pc/: 
{\xypolygon5{~={90}~*{\xypolynode}~>{}}},
"1";"2"**@{-},
"2";"3"**@{-},
"4";"5"**@{-},
\endxy $
\qquad
(7) $\xy /r2pc/: 
{\xypolygon5{~={90}~*{\xypolynode}~>{}}},
"1";"2"**@{-},
"2";"3"**@{-},
"3";"1"**@{-},
"4";"5"**@{-},
\endxy $
\end{center}
up to mutation equivalence.

(\Rnum 2) There are exactly seven point schemes
\begin{enumerate}
\item $\bigcup_{1\leq i<j<k\leq 5} \cV(x_i, x_j, x_k)$ \quad [$\ell=10$]
\item $\cV(x_1, x_4) \cup \cV(x_1, x_5) \cup \cV(x_3, x_4)\cup \cV(x_3, x_5) \cup \cV(x_1, x_2, x_3) \cup \cV(x_2, x_4, x_5)$ \quad [$\ell=2$]
\item $\cV(x_1, x_2) \cup \cV(x_1, x_4) \cup \cV(x_2, x_5) \cup \cV(x_3, x_4) \cup \cV(x_3, x_5)$  \quad [$\ell=0$]
\item $\cV(x_3, x_4) \cup \cV(x_3, x_5) \cup \cV(x_4, x_5)\cup \cV(x_1, x_2, x_3) \cup \cV(x_1, x_2, x_4) \cup \cV(x_1, x_2, x_5)$ \quad [$\ell=3$]
\item $\cV(x_5) \cup \cV(x_1, x_2) \cup \cV(x_3, x_4)$ \quad [$\ell=0$]
\item $\cV(x_1) \cup \cV(x_3) \cup \cV(x_2, x_4, x_5)$ \quad [$\ell=1$]
\item $\PP^4$ \quad [$\ell=0$]
\end{enumerate}
where $\ell$ is the number of irreducible components that are isomorphic to $\PP^1$.

(\Rnum 3) To compute $\uCM^{\ZZ}(A_{\e})$ in the case $n=5$, the graphs (1), (2) contain two distinct isolated vertices, so we can apply Lemma \ref{lem.R1}. On the other hand, the graphs (4), (5), (6), (7) contain isolated segments, so we can apply Lemma \ref{lem.R2}.
For the graph (3), we can apply Lemma \ref{lem.M2} as Example \ref{ex.M2} (1).
Consequently $\uCM^{\ZZ}(A_\e)$ is classified as follows:

$$\begin{array}{|c|c|}
\hline
(1) &  \cD^b(\mod k^{16})  \\
\hline
(2), (4), (6)  & \cD^b(\mod k^4)  \\
\hline
(3), (5), (7)  & \cD^b(\mod k)  \\
\hline
\end{array}$$

\subsubsection{The case $n=6$}
(\Rnum 1) There are exactly sixteen graphs
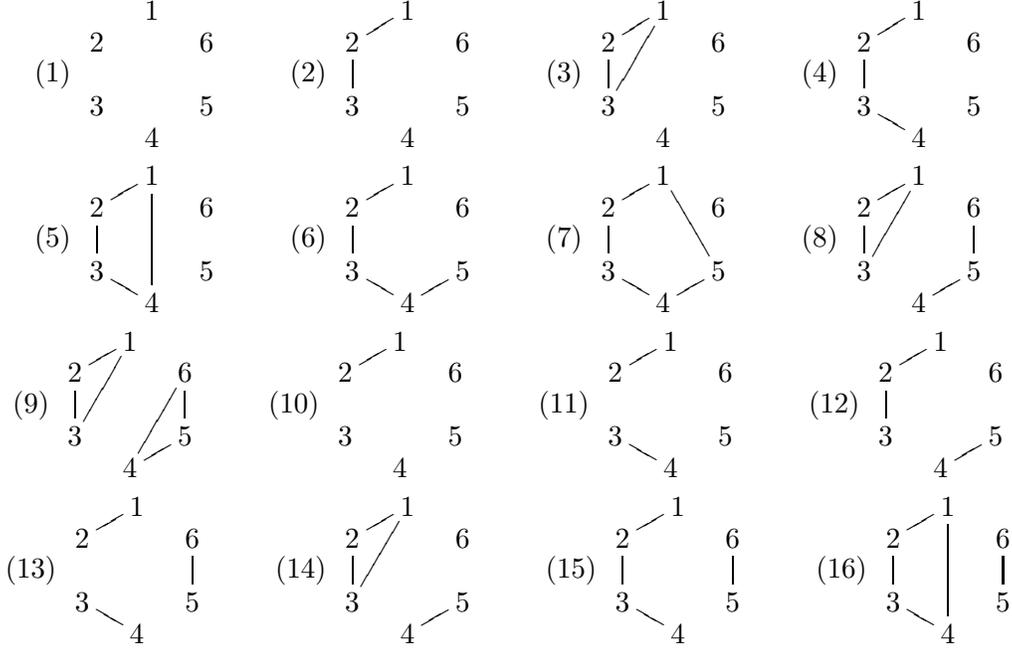
\begin{center}
(1) $\xy /r2pc/: 
{\xypolygon6{~={90}~*{\xypolynode}~>{}}},
\endxy $
\qquad 
(2) $\xy /r2pc/: 
{\xypolygon6{~={90}~*{\xypolynode}~>{}}},
"1";"2"**@{-},
"2";"3"**@{-},
\endxy $
\qquad
(3) $\xy /r2pc/: 
{\xypolygon6{~={90}~*{\xypolynode}~>{}}},
"1";"2"**@{-},
"2";"3"**@{-},
"3";"1"**@{-},
\endxy $
\qquad
(4) $\xy /r2pc/: 
{\xypolygon6{~={90}~*{\xypolynode}~>{}}},
"1";"2"**@{-},
"2";"3"**@{-},
"3";"4"**@{-},
\endxy $

(5) $\xy /r2pc/: 
{\xypolygon6{~={90}~*{\xypolynode}~>{}}},
"1";"2"**@{-},
"2";"3"**@{-},
"3";"4"**@{-},
"4";"1"**@{-},
\endxy $
\qquad 
(6) $\xy /r2pc/: 
{\xypolygon6{~={90}~*{\xypolynode}~>{}}},
"1";"2"**@{-},
"2";"3"**@{-},
"3";"4"**@{-},
"4";"5"**@{-},
\endxy $
\qquad
(7) $\xy /r2pc/: 
{\xypolygon6{~={90}~*{\xypolynode}~>{}}},
"1";"2"**@{-},
"2";"3"**@{-},
"3";"4"**@{-},
"4";"5"**@{-},
"5";"1"**@{-},
\endxy $
\qquad 
(8) $\xy /r2pc/: 
{\xypolygon6{~={90}~*{\xypolynode}~>{}}},
"1";"2"**@{-},
"2";"3"**@{-},
"3";"1"**@{-},
"4";"5"**@{-},
"5";"6"**@{-},
\endxy $

(9) $\xy /r2pc/: 
{\xypolygon6{~={90}~*{\xypolynode}~>{}}},
"1";"2"**@{-},
"2";"3"**@{-},
"3";"1"**@{-},
"4";"5"**@{-},
"5";"6"**@{-},
"6";"4"**@{-},
\endxy $
\qquad
(10) $\xy /r2pc/: 
{\xypolygon6{~={90}~*{\xypolynode}~>{}}},
"1";"2"**@{-},
\endxy $
\qquad 
(11) $\xy /r2pc/: 
{\xypolygon6{~={90}~*{\xypolynode}~>{}}},
"1";"2"**@{-},
"3";"4"**@{-},
\endxy $
\qquad 
(12) $\xy /r2pc/: 
{\xypolygon6{~={90}~*{\xypolynode}~>{}}},
"1";"2"**@{-},
"2";"3"**@{-},
"4";"5"**@{-},
\endxy $

(13) $\xy /r2pc/: 
{\xypolygon6{~={90}~*{\xypolynode}~>{}}},
"1";"2"**@{-},
"3";"4"**@{-},
"5";"6"**@{-},
\endxy $
\qquad 
(14) $\xy /r2pc/: 
{\xypolygon6{~={90}~*{\xypolynode}~>{}}},
"1";"2"**@{-},
"2";"3"**@{-},
"3";"1"**@{-},
"4";"5"**@{-},
\endxy $
\qquad
(15) $\xy /r2pc/: 
{\xypolygon6{~={90}~*{\xypolynode}~>{}}},
"1";"2"**@{-},
"2";"3"**@{-},
"3";"4"**@{-},
"5";"6"**@{-},
\endxy $
\qquad
(16) $\xy /r2pc/: 
{\xypolygon6{~={90}~*{\xypolynode}~>{}}},
"1";"2"**@{-},
"2";"3"**@{-},
"3";"4"**@{-},
"4";"1"**@{-},
"5";"6"**@{-},
\endxy $
\end{center}
up to mutation equivalence.

(\Rnum 2) There are exactly sixteen point schemes

\begin{enumerate}
\item $\bigcup_{1\leq i<j<k<l\leq 6} \cV(x_i, x_j, x_k, x_l)$\quad [$\ell=15$]
\item $\cV(x_1, x_4, x_5) \cup \cV(x_1, x_4, x_6) \cup \cV(x_1, x_5, x_6)
\cup \cV(x_3, x_4, x_5) \cup \cV(x_3, x_4, x_6)\cup \cV(x_3, x_5, x_6)
\cup \cV(x_1, x_2, x_3, x_4)\cup \cV(x_1, x_2, x_3, x_5)
\cup \cV(x_1, x_2, x_3, x_6)\cup \cV(x_2, x_4, x_5, x_6)$ \quad [$\ell=4$]
\item $\cV(x_4, x_5) \cup \cV(x_4, x_6)\cup \cV(x_5, x_6)
\cup \cV(x_1, x_2, x_3, x_4)\cup \cV(x_1, x_2, x_3, x_5)
\cup \cV(x_1, x_2, x_3, x_6)$ \quad [$\ell=3$]
\item $\cV(x_1, x_2, x_5) \cup \cV(x_1, x_2, x_6) \cup \cV(x_1, x_4, x_5)
\cup \cV(x_1, x_4, x_6) \cup \cV(x_2, x_5, x_6)\cup \cV(x_3, x_5, x_6)
\cup \cV(x_3, x_4, x_5)\cup \cV(x_3, x_4, x_6) \cup \cV(x_1, x_2, x_3, x_4)$ \quad [$\ell=1$]
\item $\cV(x_1, x_2, x_5) \cup \cV(x_1, x_2, x_6) \cup \cV(x_1, x_4, x_5)
\cup \cV(x_1, x_4, x_6) \cup \cV(x_2, x_3, x_5)\cup \cV(x_2, x_3, x_6)
\cup \cV(x_3, x_4, x_5)\cup \cV(x_3, x_4, x_6)
\cup \cV(x_1, x_2, x_3, x_4) \cup \cV(x_1, x_3, x_5, x_6) \cup \cV(x_2, x_4, x_5, x_6)$ \quad [$\ell=3$]
\item $\cV(x_3, x_6) \cup \cV(x_1, x_2, x_3) \cup \cV(x_1, x_2, x_5)
\cup \cV(x_1, x_4, x_5) \cup \cV(x_1, x_4, x_6) \cup \cV(x_2, x_5, x_6)\cup \cV(x_3, x_4, x_5)$ \quad [$\ell=0$]
\item $\cV(x_1, x_2, x_3) \cup \cV(x_1, x_2, x_5) \cup \cV(x_1, x_3, x_6) \cup \cV(x_1, x_4, x_5)
\cup \cV(x_1, x_4, x_6) \cup \cV(x_2, x_3, x_4)\cup \cV(x_2, x_4, x_6)
\cup \cV(x_2, x_5, x_6) \cup \cV(x_3, x_4, x_5)\cup \cV(x_3, x_5, x_6)$ \quad [$\ell=0$]
\item $\cV(x_4) \cup \cV(x_6) \cup \cV(x_1, x_2, x_3, x_5)$\quad [$\ell=1$]
\item $\PP^5$\quad [$\ell=0$]
\item $\cV(x_3, x_4, x_5) \cup \cV(x_3, x_4, x_6) \cup \cV(x_3, x_5, x_6) \cup \cV(x_4, x_5, x_6)
\cup \cV(x_1, x_2, x_3, x_4)\cup \cV(x_1, x_2, x_3, x_5)\cup \cV(x_1, x_2, x_3, x_6)
\cup \cV(x_1, x_2, x_4, x_5)\cup \cV(x_1, x_2, x_4, x_6)\cup \cV(x_1, x_2, x_5, x_6)$ \quad [$\ell=6$]
\item $\cV(x_5, x_6) \cup \cV(x_1, x_2, x_5) \cup \cV(x_1, x_2, x_6) \cup \cV(x_3, x_4, x_5) \cup\cV(x_3, x_4, x_6)
\cup \cV(x_1, x_2, x_3, x_4)$\quad [$\ell=1$]
\item $\cV(x_1, x_6) \cup \cV(x_3, x_6) \cup
\cV(x_1, x_2, x_3) \cup \cV(x_1, x_4, x_5) \cup \cV(x_3, x_4, x_5) \cup \cV(x_2, x_4, x_5, x_6)$\quad [$\ell=1$]
\item $\cV(x_1, x_2) \cup \cV(x_3, x_4) \cup \cV(x_5, x_6)$\quad [$\ell=0$]
\item $\cV(x_6) \cup \cV(x_4, x_5) \cup \cV(x_1, x_2, x_3)$\quad [$\ell=0$]
\item $\cV(x_1, x_2) \cup \cV(x_1, x_4) \cup \cV(x_3, x_4) \cup \cV(x_2, x_5, x_6) \cup \cV(x_3, x_5, x_6)$\quad [$\ell=0$]
\item $\cV(x_1, x_2) \cup \cV(x_1, x_4) \cup \cV(x_2, x_3) \cup \cV(x_3, x_4) \cup \cV(x_1, x_3, x_5, x_6) \cup \cV(x_2, x_4, x_5, x_6)$ \quad [$\ell=2$]
\end{enumerate}
where $\ell$ is the number of irreducible components that are isomorphic to $\PP^1$.

(\Rnum 3) 
To compute $\uCM^{\ZZ}(A_{\e})$ in the case $n=6$, the graphs (1), (2), (3), (4), (5) contain two distinct isolated vertices, so we can apply Lemma \ref{lem.R1}. On the other hand, the graphs (10), (11), (12), (13), (14), (15), (16) contain isolated segments, so we can apply Lemma \ref{lem.R2}.   By mutating at the vertex 3 in the graphs (8), (9), $[1,2]$ become isolated segments, so we can apply Lemma \ref{lem.R2}.  For the graphs (6), (7), we can apply Lemma \ref{lem.M2} as Example \ref{ex.M2} (2), (3).
Consequently $\uCM^{\ZZ}(A_\e)$ is classified as follows:

$$\begin{array}{|c|c|}
\hline
(1) &  \cD^b(\mod k^{32})  \\
\hline
(2), (3), (5), (10), (16)  & \cD^b(\mod k^8)  \\
\hline
(4), (6), (7), (8), (9), (11), (12), (13), (14), (15)  & \cD^b(\mod k^2)  \\
\hline
\end{array}$$

By the above classification, we obtain the following two theorems.

\begin{theorem}
If $n\leq 6$, then the following are equivalent: 
\begin{enumerate}
\item{} $G_{\e}$ and $G_{\e'}$ are mutation equivalent. 
\item{} $\GrMod S_{\e}\cong \GrMod S_{\e'}$. 
\item{} $E_{\e}\cong E_{\e'}$   
\end{enumerate}
\end{theorem} 

\begin{proof}  By Lemma \ref{lem.mg}, (1) $\Rightarrow$ (2), and it is well-known that (2) $\Rightarrow$ (3) in general.   
On the other hand, for each $n\leq 6$, the number of mutation equivalence classes of the graphs is 
equal to the number of isomorphism classes of the point schemes by the above classification,
so the result follows.  
\end{proof} 

\begin{theorem}
Let $\ell$ be the number of irreducible components of $E_\e$ that are isomorphic to $\PP^1$.
Assume that $n\leq 6$.
\begin{enumerate}
\item If $n$ is odd, then $\ell\leq 10$ and
\begin{align*}
\ell =0 &\Longleftrightarrow \uCM^{\ZZ}(A_\e) \cong \cD^b(\mod k),\\
0< \ell \leq 3 &\Longleftrightarrow \uCM^{\ZZ}(A_\e) \cong \cD^b(\mod k^4),\\
3< \ell \leq 10 &\Longleftrightarrow \uCM^{\ZZ}(A_\e) \cong \cD^b(\mod k^{16}).
\end{align*}
\item If $n$ is even, then $\ell\leq 15$ and
\begin{align*}
0\leq \ell \leq 1 &\Longleftrightarrow \uCM^{\ZZ}(A_\e) \cong \cD^b(\mod k^2),\\
1< \ell \leq 6 &\Longleftrightarrow \uCM^{\ZZ}(A_\e) \cong \cD^b(\mod k^8),\\
6< \ell \leq 15 &\Longleftrightarrow \uCM^{\ZZ}(A_\e) \cong \cD^b(\mod k^{32}).
\end{align*}
\end{enumerate}
In particular, \cite[Conjecture 1.3]{U} holds true for $n\leq 6$.
\end{theorem}

\begin{proof}
This follows directly by the classification. (This was proved in \cite[Theorem 1.4]{U} for $n\leq 5$.)
\end{proof}

\begin{remark} \label{rem.cex}
We can show by a counterexample that \cite[Conjecture 1.3]{U} does not hold when $n=7$.
If 
\[G_\e = \xy /r2pc/: 
{\xypolygon7{~={90}~*{\xypolynode}~>{}}},
"1";"2"**@{-},
"2";"3"**@{-},
"3";"4"**@{-},
"4";"5"**@{-},
"5";"6"**@{-},
"6";"1"**@{-},
\endxy,
\]
then
\[
\mu_{5 \leftarrow 3}\mu_{6 \leftarrow 4}\mu_{3 \leftarrow 1}\mu_{6 \leftarrow 2}(G_\e) = \xy /r2pc/: 
{\xypolygon7{~={90}~*{\xypolynode}~>{}}},
"1";"2"**@{-},
"3";"4"**@{-},
\endxy,
\]
so we have $\uCM^{\ZZ}(A_\e) \cong \cD^b(\mod k^{4})$.
However one can check that
\begin{align*}
E_\e \cong \ &\cV(x_1, x_4, x_7) \cup \cV(x_2, x_5, x_7) \cup \cV(x_3, x_6, x_7) 
\cup \cV(x_1, x_2, x_3, x_4) \cup \cV(x_1, x_2, x_3, x_6)\\
&\cup \cV(x_1, x_2, x_5, x_6)
\cup \cV(x_1, x_4, x_5, x_6) \cup \cV(x_3, x_4, x_5, x_6) \cup \cV(x_2, x_3, x_4, x_5),
\end{align*}
so $\ell =0$.
\end{remark}

\section*{Acknowledgments}
The second author thanks Ruipeng Zhu for Lemma \ref{lem.M2}, which is inspired by a communication with him.
The second author also thanks Ken Nakashima for Remark \ref{rem.cex}, which is obtained using a computer program written by him.

\end{document}